%% file: main.tex
\documentclass{amsart}

\usepackage{amsmath,amsthm,amsfonts,amscd,amssymb, tikz-cd} 
\usepackage{verbatim} 
\usepackage{mathtools}% Allows quoting source with commands.
\usepackage{epsfig}         	% Allows inclusion of eps files.
\usepackage{url}		% Allows good typesetting of web URLs.
\usepackage{epic,eepic,latexsym}
\usepackage{graphicx}
\usepackage{color}
\usepackage{lcd}
\usepackage{mathrsfs}
\usepackage[centering,text={15.5cm,22cm},
        marginparwidth=20mm,dvips]{geometry}
\usepackage[linktocpage]{hyperref}
\usepackage{lscape}
\usepackage{tabularx}
\usepackage{marginnote}
\usepackage[all]{xy}
\usepackage{enumitem}
\usepackage{stmaryrd}
\usepackage{tipa}
\usepackage{float}
\usepackage[normalem]{ulem}

\DeclareMathOperator{\Hom}{Hom}

\DeclareMathOperator{\codim}{codim}

\DeclareMathOperator{\rank}{rank}
\DeclareMathOperator{\Spec}{Spec}
\DeclareMathOperator{\id}{Id}

\DeclareMathOperator{\Tr}{Tr}

\DeclareMathOperator{\sign}{sgn}

\DeclareMathOperator{\prin}{prin}

\DeclareMathOperator{\In}{in}
\DeclareMathOperator{\Mult}{Mult}

\DeclareMathOperator{\Aut}{Aut}

\DeclareMathOperator{\scat}{Scat}

\DeclareMathOperator{\Joints}{Joints}

\DeclareMathOperator{\GL}{GL}

\DeclareMathOperator{\Supp}{Supp}

\DeclareMathOperator{\Li}{Li}

\DeclareMathOperator{\ad}{ad}
\DeclareMathOperator{\as}{as}

\DeclareMathOperator{\Parents}{Parents}
\DeclareMathOperator{\Ancestors}{Ancestors}
\DeclareMathOperator{\Leaves}{Leaves}
\DeclareMathOperator{\Flags}{Flags}
\DeclareMathOperator{\rib}{\mbox{\textramshorns}}
\DeclareMathOperator{\ob}{ob}
\DeclareMathOperator{\reg}{reg}
\DeclareMathOperator{\Hall}{Hall}

\DeclareMathOperator{\Skew}{skew}

\DeclareMathOperator{\an}{an}
\DeclareMathOperator{\sst}{ss}
\DeclareMathOperator{\scl}{cl}

\let\?=\overline
\let\u=\underline
\let\: = \colon

\let\bb=\mathbb

\let\rar=\rightarrow
\let\f=\mathfrak
\let\s=\mathcal

\let\wh=\widehat
\let\wt=\widetilde

\newcommand {\ww} {{\bf w}}

\newcommand {\JJ} {{\bf J}}

\theoremstyle{plain}% default
 \newtheorem{thm}{Theorem}[section]
 
 \newtheorem{lem}[thm]{Lemma}
  
  \newtheorem{prop}[thm]{Proposition}

\newcommand{\C}{{\mathbb C}}

\newcommand{\Q}{{\mathbb Q}}
\newcommand{\R}{{\mathbb R}}
\newcommand{\Z}{{\mathbb Z}}
\newcommand {\A} {{\bf A}}

\newcommand{\calM}{{\mathcal M}}

\newcommand{\frakd}{\mathfrak{d}}

\newcommand{\frakD}{\mathfrak{D}}

\newcommand{\St}{\operatorname{St}}

\newcommand{\rep}{\operatorname{rep}}
\newcommand{\End}{\operatorname{End}}
\newcommand{\Gr}{\operatorname{Gr}}

\newcommand{\supp}{\operatorname{supp}}

\newcommand{\Image}{\operatorname{Image}}

\newcommand{\crit}{\operatorname{crit}}

\newcommand{\enpt}{{\mathcal{Q}}}  %endpoint
\newcommand{\stab}{{\theta}}     %stability

\theoremstyle{definition}
 \newtheorem{dfn}[thm]{Definition}

 \newtheorem{eg}[thm]{Example}

\theoremstyle{remark} 
 \newtheorem{rmk}[thm]{Remark}

\title{Donaldson-Thomas invariants from tropical disks}
\author{Man-Wai Cheung}
\address{Department of Mathematics, One Oxford Street Cambridge, Harvard University, MA 02138}
\email{mwcheung@math.harvard.edu}
\author{Travis Mandel}
\address{School of Mathematics\\
University of Edinburgh\\
Edinburgh EH9 3FD\\
UK}
\email{Travis.Mandel{\char'100}ed.ac.uk}
\thanks{The second author was supported by the National Science Foundation RTG Grant DMS-1246989, and later by the Starter Grant ``Categorified Donaldson-Thomas Theory'' no. 759967 of the European Research Council.}

\begin{document}

\begin{abstract}
We prove that the quantum DT-invariants associated to quivers with genteel potential can be expressed in terms of certain refined counts of tropical disks.  This is based on a quantum version of Bridgeland's description of cluster scattering diagrams in terms of stabilitiy conditions, plus a new version of the description of scattering diagrams in terms of tropical disk counts.  The weights with which the tropical disks are counted are expressed in terms of motivic integrals of certain quiver flag varieties.  We also show via explicit counterexample that Hall algebra broken lines do not result in consistent Hall algebra theta functions, i.e., they violate the extension of a lemma of Carl-Pumperla-Siebert from the classical setting.
\end{abstract}

\maketitle

\setcounter{tocdepth}{1}
\tableofcontents

\section{Introduction}\label{Intro}

In \cite{GHKK}, Gross-Hacking-Keel-Kontsevich used scattering diagrams to construct canonical bases for cluster algebras.  Several articles \cite{Rein,GP,KSmotivic,Nagao,Kel} have developed connections between DT-invariants and various scattering diagrams or cluster transformations. Building off these ideas, Bridgeland \cite{Bridge} constructed Hall algebra scattering diagrams whose classical integrals often recover the cluster scattering diagrams (cf. our Prop. \ref{IgenteelProp} for the quantum analog).  On the other hand, \cite{GPS,CPS,FS,Man3} show how to express various scattering diagrams in terms of certain (refined) counts of tropical curves or disks.  By extending and combining these ideas, we obtain new expressions for quantum DT-invariants in terms of refined counts of tropical disks.

\subsection{Quantum DT-invariants from tropical ribbons}

Let $(Q,W)$ be a finite quiver $Q$ without loops or oriented 2-cycles, plus a choice of finite potential $W$, i.e., a finite linear combination of oriented cycles in $Q$.  Denote the vertex set of $Q$ by $Q_0$.  There are standard notions, reviewed in \S \ref{HallSection}, of the associated category of representations $\rep (Q,W)$, the corresponding Grothendieck lattice $N=\bb{Z}^{Q_0}$, and the moduli stack $\s{M}$ of objects in $\rep(Q,W)$.  Points $\stab$ in $M_{\bb{R}}:=\Hom(N,\bb{R})$ can be viewed as stability conditions on $\rep (Q,W)$, determining a substack $1_{\sst}(\stab)$ of $\stab$-stemistable objects\footnote{More precisely, $1_{\sst}(\stab)$ is the Hall algebra element associated to the substack of $\stab$-semistable objects.}  in $\rep(Q,W)$, cf. Definition \ref{semistableDef}.  We let $\s{I}_t$ denote the quantum integration map taking varieties over $\s{M}$ to elements of a quantum torus algebra $\bb{C}_t[N^{\oplus}]$, cf. \S \ref{Imap}.  E.g., if $W=0$, then $\s{I}_t$ is the generalized Poincar\'e polynomial.  We are interested are the invariants $\s{I}_t(\log(1_{\sst}(\stab)))$.

Some additional notation and terminology regarding this setup will be needed.  The lattice $N$ is equipped with the natural basis $\{e_i\}_{i\in Q_0}$.  Let $B$ denote the skew-symmetric Euler form on $N$, cf. \eqref{Bdef}, and define\footnote{In \cite{GHKK}, $p^*$ is defined by inserting $n$ into the first entry of the skew-symmemtrizable form.  This discrepancy is because the form $B$ here and in \cite{Bridge} is negative the form used in \cite{GHKK}.} $p^*:N\rar M$, $p^*(n)=B(\cdot,n)$.  We say $\stab\in M_{\bb{R}}$ is general if it is not in the intersection of two distinct hyperplanes of the form $n^{\perp}$ for $n\in N\setminus \{0\}$, cf. Remark \ref{general}.

We will also need the following setup coming from the theory of tropical curves and scattering diagrams, cf. \S \ref{TropDisks} for details.  By a weight-vector, we will mean a tuple $\ww=(\ww_i)_{i\in Q_0}$ where each $\ww_i=(w_{ij})_{j=1,\ldots,l_i}$ consists of positive integers $w_{i1} \leq w_{i2} \leq \ldots \leq w_{il_i}$.  Denote the length $l(\ww):=\sum_i l_i$, and let $\Aut(\ww)$ be the group of automorphisms of the second indices of the $\ww_i$'s which act trivially on $\ww$.  Define the ``multiple cover contributions''  $R_{\ww}\coloneqq \prod_{ij} \frac{(-1)^{w_{ij}-1}}{w_{ij}(q^{w_{ij}}-1)}\in \bb{C}[t^{\pm 1}]$ where $q\coloneqq t^2$.

One says that a tropical disk $h:\Gamma\rar M_{\bb{R}}$ (cf. \S \ref{TropDisks}) has degree $\Delta_{\ww}$ if the unbounded edges $E_{ij}$ are labelled by the indices of $\ww$, and if the weighted outgoing direction of $h(E_{ij})$ equals $w_{ij}p^*(e_i)$.  Let $\A_{\ww}$ be a collection of affine hyperplanes $\{A_{ij}\subset M_{\bb{R}}\}$ with $A_{ij}$ a generic translate of $e_i^{\perp}$.  For $\delta>0$, we say that a tropical disk matches the constraints $\delta\A_{\ww}$ if $h(E_{ij})\subset \delta A_{ij}$ for each $i,j$.  The type $\tau$ of a tropical disk is the data of the underlying weighted graph $\Gamma$ plus the data of directions of $h(E)$ for each edge $E$ of $\Gamma$.

Each tropical disk includes the data of a special endpoint vertex $V_{\infty}\in \Gamma^{[0]}$, and we will impose an additional constraint on the image of $V_{\infty}$.  Specifically, given $\stab\in M_{\bb{R}}$ and a tropical disk type $\tau$, we say that $\tau \in \f{T}_{\ww}(\stab)$ if the following holds: given $\epsilon>0$ and any sufficiently small $\delta>0$ (small relative to $\epsilon$), there exist tropical disks of degree $\Delta_{\ww}$ and type $\tau$ which match the constraint $\delta \A_{\ww}$ and have $h(V_{\infty})\in B_{\epsilon}(\stab)$ (the radius $\epsilon$ open ball about $\stab$).  Let $\wh{\f{T}}_{\ww}(\stab)$ denote the corresponding space of tropical ribbon types, i.e., tropical disk types plus the additional data of a cyclic ordering of the edges at each vertex.  We wish to count elements of $\wh{\f{T}}_{\ww}(\stab)$ with a multiplicity which we define next.

The form $B$ descends to a form $\?{B}$ on $p^*(N)$ given by $\?{B}(p^*(n_1),p^*(n_2))=B(n_1,n_2)$.  For $\wh{\tau}\in \wh{\f{T}}_{\ww}(\stab)$, let $\nu(\wh{\tau})$ denote $(-1)$ to the power of the number of vertices of $\wh{\tau}$ where the ribbon structure does not agree with the orientation induced by $\?{B}$, cf. \S \ref{MultSect}.  The ribbon structure induces an ordering $E_{i_1,j_1},\ldots E_{i_{l(\ww)}j_{l(\ww)}}$ on the $E_{ij}$'s.  Given such a tropical ribbon type $\wh{\tau}$, let $\f{Flag}(\wh{\tau})$ denote the variety over $\s{M}$ whose fiber over a (stacky) point corresponding to a representation $M$ is the space of composition series
 \begin{align*}
     0=M_0 \subset M_1 \subset \ldots \subset M_{\sum_{ij} w_{ij}} = M
 \end{align*}
such that the first $w_{i_1j_1}$ quotients $M_i/M_{i-1}$ are isomorphic to the simple representation $S_{i_1}$, then the next $w_{i_2j_2}$ quotients $M_i/M_{i-1}$ are isomorphic to the simple representation $S_{i_2}$, and so on.  The following is the quantum integral case of Theorem \ref{Main-DT-Trop-Thm}.

\begin{thm}\label{MainIntro}
Suppose $(Q,W)$ is genteel over $\f{g}^q$ (cf. \S \ref{GenteelSect}).  Then for general $\stab\in M_{\bb{R}}$,
\begin{align}\label{MainEqnIntro}
    \s{I}_t(\log(1_{\sst}(\stab))) = \sum_{\ww}\left( \frac{1}{|\Aut(\ww)|}\sum_{\wh{\tau}\in \wh{\f{T}}_{\ww}(\stab)} \nu(\wh{\tau})\s{I}_t(R_{\ww}\f{Flag}(\wh{\tau}))\right).
\end{align}
\end{thm}

See \S \ref{GenteelSect} for details on our version of the genteel property and modifications thereof.  We note here that genteelness of $(Q,W)$ is known to at least hold for acyclic quivers with $W=0$.  More generally, the possibly weaker condition of genteelness over $\f{g}^q$ holds whenever $Q$ admits a green-to-red sequence and $W$ is non-degenerate \cite[Cor. 1.2(i)]{Mou}.

The same statement applies with $\s{I}_t$ replaced by the classical integration map $\s{I}$ (i.e., the $t\mapsto 1$ limit, i.e., taking generalized Euler characteristics), and similarly for $\s{I}_t$ replaced by other projections $\s{I}^{\f{i}}$ of the Hall algebra defined in \S \ref{HallScatSetup}.  For the classical version though, one should view $\s{I}(R_{\ww}\f{Flag}(\wh{\tau}))$ as living in $\f{A}^{\scl}$, a logarithmic version of the Weyl algebra, cf. Example \ref{fA}.  These factors $\s{I}_t(R_{\ww}\f{Flag}(\wh{\tau}))$ and $\s{I}(R_{\ww}\f{Flag}(\wh{\tau}))$ can be more easily computed as products in the quantum torus algebra or Weyl algebra, respectively, cf. Remark \ref{MultComputation}.

Alternatively, one can replace the sum over tropical ribbons with a sum over tropical disks, and then the tropical ribbon multiplicities $\nu(\wh{\tau})\s{I}_t(R_{\ww}\f{Flag}(\wh{\tau}))$ are replaced with Block-G\"ottsche \cite{BG} style refined tropical disk multiplicities $R'_{\ww}\prod_V [\Mult(V)]_t$, and similarly for the classical cases, cf. Remark \ref{MultComputation} again.  Our intention in this paper though is to give a representation-theoretic description of the tropical multiplicities, which is why we state Theorem \ref{MainIntro} in terms of moduli of flags.

See Example \ref{TermComputation} for a sample computation of a term on the right-hand side of \eqref{MainEqnIntro}.

\subsection{Hall algebra broken lines violate the Carl-Pumperla-Siebert Lemma}

One might hope (as we had hoped) that Theorem \ref{MainIntro} holds without applying the integration maps, i.e., as an identity in the Hall algebra.  Unfortunately, this fails as a result of the fact that elements of the Hall algebra with parallel dimension vectors need not commute (although we see that the result does hold after modding out by the ideal generated by these commutators).  In \S \ref{ThetaSection}, we show that similar issues cause problems for theta functions.

As in \cite{CPS,GHK1,GHS}, the construction of theta functions in \cite{GHKK} is based on enumerating broken lines (an abridged version of tropical disks).  This enumeration depends on the designated endpoint of the broken lines, but according to \cite[\S 4]{CPS}, different choices of endpoint are related by path-ordered product, essentially meaning that these choices glue to give well-defined global functions on the mirror.  \cite[Thm. 2.14]{Man3} gives a refined version of this lemma of Carl-Pumperla-Siebert, implying that the analogous gluing property holds for quantum theta functions, cf. Lemma \ref{CPS}.  Refining further, \cite{mthesis, cheung2019theta} defines Hall algebra broken lines, and from these one might hope to define Hall algebra theta functions.  Unfortunately, this is not a well-behaved notion:
\begin{prop}[Prop. \ref{CPSfail} in the main text]
The Carl-Pumperla-Siebert Lemma does not hold for Hall algebra broken lines.
\end{prop}
Our proof is via the explicit construction of a  counterexample for an $A_3$-quiver, cf. \S \ref{CPS-counterexample-section}.

\subsection{Motivation}
When $B$ has rank $2$, the tropical disk counts of Theorem \ref{MainIntro} can be replaced with tropical curve counts, cf. \cite[Thm. 2.8]{GPS} and \cite[Cor. 4.9]{FS}.  In higher-dimensions this is only the case for certain limits of choices of $\stab$, cf. \cite[Thm 3.7]{Man3}.  The tropical curve versions are nice because in the classical limit they can be related via \cite{NS} to log Gromov-Witten invariants, cf. \cite[Prop. 5.3]{GPS}.  As the authors have learned from Mark Gross, the classical versions of our tropical disk counts should also have an algebraic Gromov-Witten theoretic meaning: according to the announced result \cite[Thm. 2.14]{GSInt}, they should be related to certain \textit{punctured} Gromov-Witten invariants (one also expects the existence of correspoding holomorphic disk counts defined from the perspective of open Gromov-Witten theory, e.g., as in \cite{lin2017correspondence} for the case of K3 surfaces).  One expects DT/GW correspondence results to follow from \cite[Thm. 2.14]{GSInt} combined with \cite[Lem. 11.4]{Bridge}.

On the other hand, the quantum tropical curve counts in rank $2$ are Block-G\"ottsche invariants \cite{BG}, which have been related to higher-genus Gromov-Witten invariants in \cite{Bou,Bou2} and to real curve counts in \cite{Mikq}.  Upcoming work of the second author will extend the correspondence of \cite{Mikq} to higher-dimensions, although an extension to tropical disks is still more distant.  Still, we hope that Theorem \ref{MainIntro} will lead to new refined DT/GW correspondence results, and we further hope that this correspondence will be enriched by our interpretation of tropical ribbon multiplicities in terms of moduli of composition series.

A version of Theorem \ref{MainIntro} for bipartite quivers was previously observed in \cite[Thm. 5.3]{FS}.  Their argument was based on the observation that in these cases, \eqref{MainEqnIntro} is equivalent to a representation-theoretic formula of Manschot-Pioline-Sen \cite{MPS}.  We therefore hope that our result may be related to some generalization of this MPS formula.  With this in mind, we strongly suspect that our tropical counts are closely related to the attractor flow trees studied by physicists, cf. \cite{AP} in particular, as well as \cite{WCS}.

We note that \cite{LMY}, which appeared immediately after this paper was first posted, deals with similar problems on scattering diagrams and tropical disks using a differential-geometric perspective.

\subsection{Outline of the paper}

In \S \ref{Prelim}-\ref{HallConstruction}, we review Joyce's construction \cite{joyce2007configurations} of the Hall algebra associated to a quiver with potential, following \cite[\S 4-5]{Bridge}. Then in \S \ref{kfold}-\ref{compalg}, we use \cite[Lem. 4.4]{BrIntro} to describe certain products in the Hall algebra in terms of moduli of composition series.  We review the quantum and classical integration maps in \S \ref{qtor}-\ref{Imap}.

We review the definition of scattering diagrams in \S \ref{Scattering}, and in Theorem \ref{KSGS} we generalize previously known results about initial scattering diagrams uniquely determining consistent scattering diagrams.  We then we review Bridgeland's Hall algebra scattering diagrams (and some variants) in \S \ref{HallScat}.  If the potential $W$ is genteel, then the Hall algebra scattering diagram is determined by an easily understood initial scattering diagram which we describe explicitly in \S \ref{HallInitialSection}.

We review the notion of tropical disks in \S \ref{TropDisks}, and in \S \ref{ScatTrop} we focus on the tropical ribbons and multiplicities associated to an initial scattering diagram.  The description of scattering diagrams in terms of tropical disks (Theorem \ref{TropicalScattering}) is given in \S \ref{ScatFromRib} and proven in \S \ref{TropScatProof}.  This is applied to the Hall algebra scattering diagram in \S \ref{MainThmSect} to prove our main results, Theorems \ref{HallScatThm} and \ref{Main-DT-Trop-Thm}.

We turn our attention to theta functions in \S \ref{ThetaSection}.  We review the definitions of broken lines and theta functions in \S \ref{brokentheta}, explaining how these apply to various flavors of cluster varieties in \S \ref{ClusterFlavors}.  Finally, in \S \ref{CPS-counterexample-section}, we work out an explicit counterexample to show that a foundational result of \cite{CPS} (cf. Lemma \ref{CPS}) does not extend to the Hall algebra setting (Proposition \ref{CPSfail}).

\subsection*{Acknowledgements}
The authors are very grateful to Ben Davison for patiently and repeatedly explaining the definition of the quantum integration map and for checking the corresponding part of our draft.  We also thank Tom Bridgeland, Lang Mou, and Tom Sutherland for helpful conversations.  Additionally, we are grateful to the anonymous referee for suggesting numerous improvements.

\section{The motivic Hall algebra of a quiver with potential}\label{HallSection}

\subsection{Preliminaries on quivers with potential and their representations}\label{Prelim}

Let $Q$ be a finite quiver.  Denote the sets of vertices and arrows of $Q$ as $(Q_0, Q_1)$.   Let $\bb{C}Q$ denote the path algebra of $Q$. 
Suppose that $Q$ is equipped with a finite \textbf{potential}, i.e., a finite linear combination of cycles, denoted $W \in \C Q$.
Define a two-sided ideal $I_W \subseteq \C Q$ on $Q$ by
\[
I_W = ( \partial_a W: a \in Q_1 ).
\]
Here, if $b_1\ldots b_k$ is a cycle of arrows in $Q$, then $$\partial_a(b_1\ldots b_k) = \sum_{i=1}^k \delta_{ab_i}b_{i+1}\ldots b_kb_1\ldots b_{i-1},$$
where $\delta_{ab_i}$ is $1$ if $a=b_i$ and $0$ otherwise.  Then the Jacobi algebra for $(Q,W)$ is the quotient algebra $\C Q/I_W$.  Let $\rep (Q,W) \coloneqq  \bmod \C Q/I_W$ be the abelian category of finite-dimensional representations of the quiver with potential $(Q,W)$, i.e., finite-dimensional left $\C Q/I_W$-modules.

Set $N= \Z^{Q_0}$, $M = \Hom_{\Z} (N, \Z)$, $M_{\R} = M \otimes_{\Z} \R$.  Let $\{e_i\}_{i \in Q_0}$ be the natural basis indexed by the vertices of $Q$.  Denote $N^{\oplus} \coloneqq  \{ \sum_i a_ie_i \in N | a_i \in \bb{Z}_{\geq 0} \ \forall  i\}$, and $N^+\coloneqq N^{\oplus}\setminus \{0\}$.
There is a group homomorphism 
\[\dim : K_0(\rep(Q,W)) \rightarrow N\]
sending a representation to its dimension vector.  For vertices $i,j\in Q_0$, let $a_{ij}$ denote the number of arrows from $i$ to $j$.  Let $B$ denote the integral skew-symmetric bilinear form on $N$ determined by setting%\footnote{Our $B$ is negative the pairing $\langle\cdot,\cdot\rangle$ used in \cite{Bridge}.  We also have a sign difference in our definition of incoming walls in Def. \ref{WallDfn}.  These two sign differences cancel to make our conventions equivalent.}
\begin{align}\label{Bdef}
    B(e_i,e_j) \coloneqq {a_{ji}-a_{ij}}.
\end{align}
We note that our $B$ is negative the pairing $\langle \cdot,\cdot \rangle$ used in \cite{Bridge}.  We will also use a second $\bb{Z}$-valued bilinear form $\chi$ on $N$ given by
\begin{align}\label{Euler}
    \chi(e_i,e_j)\coloneqq \delta_{ij}-a_{ij}.
\end{align}
Note that $B(a,b)\coloneqq {\chi(a,b)-\chi(b,a)}$.

It is well-known (cf. \cite[Lem 4.1]{Bridge}) that there is an algebraic moduli stack $\calM$ parameterizing all objects of the category $\rep(Q,W)$.  Briefly, objects of $\s{M}$ over a scheme $S$ are isomorphism classes of locally free finite-rank $\s{O}_S$-modules $E$, together with morphisms $\rho:\bb{C}Q/I_W \rar \End_S(E)$, cf. \cite[\S 4.2]{Bridge} for details.  Furthermore, $\calM$ decomposes as 
\begin{align}\label{decomp}
\calM=\bigsqcup_{d\in N^{\oplus}} \calM_d 
\end{align}
where $\calM_d$ is the open and closed substack parametrizing objects of dimension vector $d$. 
There is a 2-category of algebraic stacks over $\calM$, and we let $\St/ \calM$ denote the full subcategory consisting of objects $f : X \rightarrow \calM$ for which $X$ is of finite type over $\Spec \C$ and has affine stabilizers.  We similarly write $\St/\C$ for the analogous category of stacks over $\Spec \C$.

\subsection{Construction of the Hall algebra}\label{HallConstruction}
We now review the motivic Hall algebra developed by Joyce \cite{joyce2007configurations}, following the presentation of \cite[\S 5]{Bridge}. 

Let $K(\St/ \calM)$ be the free abelian group with basis given by isomorphism classes of objects of $\St / \calM$ modulo the relations given in \cite[Def. 5.1]{Bridge}.  In particular, one imposes the scissor relations
\begin{align*}
[f:X\rar \s{M}]=[f|_Y:Y\rar \s{M}]+[f|_U:U\rar \s{M}],  
\end{align*}
where $[f:X\rar \s{M}]$ is an object of $\St/\s{M}$, $Y\subset X$ is a closed substack, and $U\coloneqq X\setminus Y$.

One endows the group $K(\St/ \calM)$ with a $K(\St/ \C)$-module structure by setting $[X] \cdot [Y \rightarrow \calM ] = [ X \times Y \rightarrow \calM]$ and extending linearly. There is a unique ring homomorphism 
\begin{align}\label{Ups}
    \Upsilon:K(\St/ \C) \rar \bb{C}(t)
\end{align}
taking the class of a smooth projective variety $X$ over $\bb{C}$ to its Poincar\'e polynomial $$\sum_{k=1}^{2d} \dim_{\bb{C}} H^k(X_{\an},\bb{C})(-t)^k\in \bb{C}[q],$$ where $q\coloneqq t^2$ and $H^k(X_{\an},\bb{C})$ denotes singular cohomology.  For $X\in K(\St/\C)$, we will often denote
\begin{align*}
    |X|\coloneqq \Upsilon(X).
\end{align*}
Let 
\begin{align*}
    K_{\Upsilon}(\St/ \calM)\coloneqq  K(\St/ \calM)\otimes_{K(\St/ \C)} \bb{C}(t).
\end{align*}

As a $\bb{C}(t)$-module, the (motivic) \textbf{Hall algebra} $H(Q,W)$ is $K_{\Upsilon}(\St/ \calM)$.  To define the multiplication, the convolution product, on $H(Q,W)$ and make it into a $\bb{C}(t)$-algebra, we consider the stack $\calM^{(2)}$ of short exact sequences in $\rep(Q,W)$. 
There is a diagram
\begin{equation}\label{m2}
\begin{tikzcd}
\calM^{(2)} \arrow[d, "{(a_1,a_2)}"] \arrow[r, "b"]  & \calM \\
\calM \times \calM,  &
\end{tikzcd}
\end{equation}
where $a_1, a_2, b$ sends a short exact sequence 
\[
0 \rightarrow A_1 \rightarrow B \rightarrow A_2 \rightarrow 0
\]
to $A_1$, $A_2$, and $B$ respectively.
The convolution product is defined to be
\[
m = b_* \circ (a_1, a_2)^* : H(Q,W) \times H(Q,W) \rightarrow H(Q,W).
\]
This product can be expressed as
\[
[X_1 \xrightarrow{f_1} \calM] * [X_2 \xrightarrow{f_2} \calM] = [Z \xrightarrow{ b \circ h} \calM],
\]
where $Z$ and $h$ are defined by the Cartesian square
\[\begin{tikzcd}
Z              \arrow[d] \arrow[r, "h"]    & \calM^{(2)} \arrow[d, "{(a_1, a_2)}"] \arrow[r, "b"] & \calM \\
X_1 \times X_2 \arrow[r, "f_1 \times f_2"] & \calM \times \calM                               &
\end{tikzcd}\]

The following is due to Joyce \cite[Thm. 5.2]{joyce2007configurations}, see also \cite[Thm. 4.3]{BrIntro}.

\begin{thm}
The product $m$ gives $H(Q,W)$ the structure of an associative unital algebra over $\bb{C}(t)$.  
The unit element is $1 = [\calM_0 \subset \calM]$.
\end{thm}

We note that the decomposition \eqref{decomp} of $\s{M}$ induces an $N^{\oplus}$-grading
\begin{align}\label{Hgrade}
H(Q,W)=\bigoplus_{d\in N^{\oplus}} H(Q,W)_d,
\end{align}
where $H(Q,W)_d$ is the submodule of $K_{\Upsilon}(\St/\s{M})$ generated by objects of the form $[X\rar \s{M}_d\subset \s{M}]$.

\subsection{$k$-fold products}\label{kfold}

We will also need a description of the $k$-fold product $m_k:H(Q,W)^{\otimes k} \rar H(Q,W)$.  For this we follow \cite[\S 4.1-4.2]{BrIntro}.  Let $\s{M}^{(k)}$ denote the algebraic moduli stack of $k$-flags.  That is, the objects of $\s{M}^{(k)}$ over a scheme $S$ are isomorphism classes of $k$-tuples of objects $(E_1,\rho_1),\ldots,(E_k,\rho_k)$ of $\s{M}(S)$, together with monomorphisms
\begin{align}\label{kFlag}
0=E_0\hookrightarrow E_1 \hookrightarrow \cdots \hookrightarrow E_k
\end{align}
respecting the maps $\rho_i$ and such that each factor $F_i\coloneqq E_i/E_{i-1}$ is flat over $S$.  Given another scheme $T$, an object $(E_1',\rho_1'),\ldots,(E_k',\rho_k')$ over $T$, and a morphism $f:T\rar S$, a morphism in $\s{M}^{(k)}$ lying over $f$ is a collection of isomorphisms of sheaves $\Phi_i:f^*(E_i)\rar E_i'$ respecting the maps $\rho_i$ and the maps in the sequences of monomorphisms as in \eqref{kFlag}.

For each $i=1,\ldots,k$, we have a morphism of stacks $a_i:\s{M}^{(k)} \rar \s{M}$ taking an object as in \eqref{kFlag} to its $i$-th factor $F_i=E_i/E_{i-1}$.  We also have another morphism $b:\s{M}^{(k)} \rar \s{M}$ taking the object as in \eqref{kFlag} to the final term $(E_k,\rho_k)$ of the sequence.  One easily sees that the stack $\s{M}^{(2)}$, together with these morphisms $a_1,a_2,b$, is equivlaent to the data we had when defining $\s{M}^{(2)}$ as the stack of short exact sequences above.  We now obtain a diagram generalizing \eqref{m2}:
\[
\begin{tikzcd}
\calM^{(k)} \arrow[d, "{(a_1,\ldots,a_k)}"] \arrow[r, "b"]  & \calM \\
\calM^{k}  &
\end{tikzcd}
\]
\begin{lem}[\cite{BrIntro}, Lemma 4.4]\label{mkLem}
The \textbf{$k$-fold product} $m_k:H(Q,W)^{\otimes k} \rar H(Q,W)$ is given by
\begin{align*}
    m_k\coloneqq b_*\circ (a_1,\ldots,a_k)^*.
\end{align*}
\end{lem}

\subsection{$H_{\reg}$ and the composition algebra}\label{compalg}

Next, recalling the notation $q=t^2$, let \[\bb{C}_{\reg}(t)\coloneqq \bb{C}[t,t^{-1}][(1+q+q^2+\ldots+q^k)^{-1}:k\geq 1]\subset \bb{C}(t).\]
Let $H_{\reg}(Q,W)$ be the $\bb{C}_{\reg}(t)$-submodule of $H(Q,W)$ generated by elements of the form $$[f: X \to \calM]$$ such that $X$ is a variety over $\C$ (so in particular, $X\in \St/\C$, and so we can apply $\Upsilon$ to $X$). 
\begin{lem}[\cite{Bridge}, Thm. 5.2]\label{regLem} $H_{\reg}(Q,W)$ is closed under the Hall algebra product and thus forms an $N^{\oplus}$-graded $\C_{\reg}(t)$-subalgebra.  Furthermore, $H_{\reg}(Q,W)$ forms a Poisson algebra under the bracket 
\begin{align}\label{regPois}
    \{a,b\}\coloneqq (t-t^{-1})^{-1}[a,b].
\end{align}
\end{lem}

Now, for any representation $A\in \ob(\rep(Q,W))$, let $p_A$ denote the corresponding (stacky) point in $\s{M}$, and let $\delta_A$ be the element of $H(Q,W)$ corresponding to the inclusion $[p_A\hookrightarrow \s{M}]$.  Let
\begin{align}\label{kappaDef}
    \kappa_A\coloneqq |\Aut(A)|\delta_A\in H(Q,W)
\end{align}
be the element $[\Spec \bb{C}\rar p_A\in \s{M}]$.  Clearly, $\kappa_A$ is in $H_{\reg}(Q,W)$. 

Given a collection of objects $A_1,\ldots,A_k,M\in \ob(\rep(Q,W))$, let $\wt{\f{Flag}}(A_1,\ldots,A_k;M)$ denote the space of filtrations
 \begin{align*}
     0=M_0 \subset M_1 \subset \ldots \subset M_k = M
 \end{align*}
 of $M$ such that $M_i/M_{i-1}\cong A_i$ for each $i$.  We also consider the quotient stack $\f{Flag}(A_1,\ldots,A_k;M)$ in which the identification of $M_k$ with $M$ is no longer part of the data of an object.  This has the effect of enlarging the automorphism groups since now automorphisms of $M$ induce automorphisms of flags, so
 \begin{align}\label{Flag-wtFlag}
    \f{Flag}(A_1,\ldots,A_k;M)= [\wt{\f{Flag}}(A_1,\ldots,A_k;M)/\Aut(M)].
 \end{align}
By Lemma \ref{mkLem}, we have the following:
\begin{lem}\label{CompSeriesLem}
Given a collection of objects $A_1,\ldots,A_k\in \ob(\rep(Q,W))$, let $d=\sum_{j=1}^k \dim(A_j) \in N^{\oplus}$.  The product $\kappa_{A_1}\cdots \kappa_{A_k}$ is represented by a complex variety $\f{Flag}(A_1,\ldots,A_k)\rar \s{M}_d$ whose fiber over a point $p_M$ is $\wt{\f{Flag}}(A_1,\ldots,A_k;M)$.  Equivalently, the fiber of $\f{Flag}(A_1,\ldots,A_k)$ over the geometric point $[\Spec \bb{C} \rar p_M\in \s{M}_d]$ is $\f{Flag}(A_1,\ldots,A_k;M)$.
\end{lem}

For each vertex $i\in Q_0$, we have an associated simple representation $S_i\in \rep(Q,W)$ of dimension vector $e_i$.  We denote $\delta_i\coloneqq \delta_{S_i}$ and $\kappa_i\coloneqq \kappa_{S_i}$.  More generally, for each $k\in \bb{Z}_{\geq 0}$, we will write the semisimple representation $S_{i}^{\oplus k}$ as $S_{ki}$, and we will write $\delta_{ki}\coloneqq \delta_{S_{ki}}$ and $\kappa_{ki}\coloneqq \kappa_{S_{ki}}$.   As in \cite[Ex. 5.20]{joyce2007configurations}, we define the \textbf{composition algebra} $\s{C}(Q,W)$ to be the subalgebra of $H_{\reg}(Q,W)$ generated by the elements $\kappa_{i}$ for $i\in Q_0$.  By Lemma \ref{CompSeriesLem}, products of the elements $\kappa_i$ are given in terms of spaces of composition series.

\begin{eg}\label{kappai}
For $i\in Q_0$, let us apply Lemma \ref{CompSeriesLem} to $\kappa_{i}^k$.  The only point in $\s{M}_{ke_i}$ is the one corresponding to the semisimple representation $S_{ki}$.  Furthermore, $\f{Flag}(S_i,\ldots,S_i;S_{ki})$ ($S_i$ occurring $k$ times before the semicolon) contains only one (stacky) point---all maximal flags of $S_{ki}$ are related by automorphisms of $S_{ki}$.  The stabilizer group for this point (i.e., the space of automorphisms of $\bb{C}^k$ which fix a maximal flag) is the unipotent group $U_k(\C)$.  Thus, 
 \begin{align}\label{kappaik}
     \kappa_{ki} = |U_k(\C)| \kappa_i^k = q^{k(k-1)/2} \kappa_i^k.
 \end{align}
 Using \eqref{kappaDef} and the fact that
 \begin{align*}
     |\Aut(S_{ki})| = |\GL_k(\C)| =q^{k(k-1)/2}\prod_{j=1}^k (q^j-1),
 \end{align*}
we can re-express \eqref{kappaik} as
\begin{align}\label{delta-ki-kappa}
    \delta_{ki}&=\frac{1}{\prod_{j=1}^k (q^j-1)} \kappa_i^k.
\end{align}
Alternatively, this could be realized directly as $$\delta_{ki}=\kappa_i^k/|\wt{\f{Flag}}(S_i,\ldots,S_i;S_{ki})|$$ ($S_i$ again appearing $k$ times before the semicolon). \eqref{delta-ki-kappa} will be useful in \S \ref{HallInitialSection}.  
 \end{eg}

\subsection{The quantum torus algebra}\label{qtor}

Let $\bb{C}_t[N^{\oplus}]$ denote the \textbf{quantum torus algebra}, by which we mean the $N^{\oplus}$-graded algebra defined by:
\begin{align*}
    \bb{C}_t[N^{\oplus}]\coloneqq \bb{C}_{\reg}(t)[z^n:n\in N^{\oplus}]/\langle z^{n_1}z^{n_2}=t^{B(n_1,n_2)}z^{n_1+n_2}:n_1,n_2\in N^{\oplus}\rangle
\end{align*}
(the monomials $z^n$ adjoined here are non-commuting).  This forms a \textbf{Poisson algebra} under the bracket
\begin{align}\label{qPois}
    \{a,b\}\coloneqq \frac{[a,b]}{t-t^{-1}}.
\end{align}
Note that
\begin{align*}
    \{z^{n_1},z^{n_2}\}=[B(n_1,n_2)]_tz^{n_1+n_2},
\end{align*} 
where for any $a\in \bb{Z}$,
\begin{align}\label{at}
    [a]_t\coloneqq \frac{t^{a}-t^{-a}}{t-t^{-1}} = \sign(a)(t^{-|a|+1} + t^{-|a|+3} + \ldots + t^{|a|-3} + t^{|a|-1}).
\end{align}
The usual commutative algebra $\C[N^{\oplus}]$ also forms a Poisson algebra, with bracket defined by 
\begin{align}\label{Pois}
    \{z^{n_1},z^{n_2}\}\coloneqq B(n_1,n_2)z^{n_1+n_2}.
\end{align}
Note that there is a surjective homomorphism of Poisson algebras defined by 
\begin{align*}
\pi_{t\mapsto 1}:\bb{C}_t[N^{\oplus}]\rar \bb{C}[N^{\oplus}], \quad  t\mapsto 1,~ z^n \mapsto z^n.    
\end{align*}
\begin{rmk}\label{PvC}
Note that $\bb{C}_t[N^{\oplus}]$ viewed as a Lie algebra with its Poisson bracket is isomorphic as a Lie algebra to $(t-t^{-1})^{-1}\cdot \bb{C}_t[N^{\oplus}]$ with its commutator bracket via the map $$x\mapsto \frac{x}{t-t^{-1}}.$$
We may thus view $\pi_{t\mapsto 1}$ as a Lie algebra homomorphism $(t-t^{-1})^{-1}\cdot \bb{C}_t[N^{\oplus}]\rar \bb{C}[N^{\oplus}]$. 
Similarly, as noted in \cite[\S 5.9]{Bridge}, $H_{\reg}(Q,W)$ with the bracket from \eqref{regPois} is isomorphic as a Lie algebra to $(t-t^{-1})^{-1}\cdot H_{\reg}(Q,W)$ with its commutator bracket.  In \S \ref{Imap}, we will discuss the ``integration map'' $\s{I}=\pi_{t\mapsto 1}\circ \s{I}_t$ as a homomorphism of Poisson algebras $H_{\reg}(Q,W)\rar \bb{C}[N^{\oplus}]$, but this can also be viewed as a homomorphism of Lie algebras $(t-t^{-1})^{-1}\cdot H_{\reg}(Q,W)\rar \bb{C}[N^{\oplus}]$.  Similarly, we may view the quantum integration map $\s{I}_t:H_{\reg}(Q,W)\rar \bb{C}_t[N^{\oplus}]$ as a Lie algebra homomorphism $(t-t^{-1})^{-1}\cdot H_{\reg}(Q,W) \rar (t-t^{-1})^{-1} \bb{C}_t[N^{\oplus}]$.

In place of the quantum torus algebra $\C_t[N^{\oplus}]$ considered above, one may use the quantum tropical vertex group of \cite[\S 6.1]{KScoha} or the quantum torus Lie algebra of \cite[\S 2.2.3]{DMan}.  These alternatives are nice because they still admit well-defined Poisson algebra maps $\pi_{t\mapsto 1}$ to $\C[N^{\oplus}]$, but now the Poisson bracket for the domain is simply the commutator bracket.  While this is often convenient, we shall not use this viewpoint here.
\end{rmk}

\subsection{The integration map}\label{Imap}

There are several constructions of \textbf{(quantum) integration maps} in the literature, i.e., homomorphisms (of algebras, Lie algebras, or Poisson algebras) from $H(Q,W)$ or $H_{\reg}(Q,W)$ to the (quantum) torus algebra.  Reineke \cite[Lem. 6.1]{ReinHN} first constructed the analog of such a quantum integration map for finitary Hall algebras associated to quivers without potential.  Joyce \cite[\S 6]{joyce2007configurations} then constructed classical and quantum integration maps with domain $H_{\reg}(Q,0)$.  The classical version of Joyce's map (of Lie algebras) was generalized to quivers with potential in \cite[\S 7]{JS} (cf. \cite[Thm. 11.1]{Bridge} for an interpretation as a map of Poisson algebras).  On the other hand, a very general construction of algebra homomorphisms from a full Hall algebra to the ``motivic quantum torus algebra'' (which can then be further integrated to the usual quantum torus algebra)  has been outlined by Kontsevich and Soibelman \cite[\S 6]{KSmotivic}.  Making this more precise and more algebraic, in \cite[\S 7]{KScoha}, Kontsevich and Soibelman defined a (monodromic) mixed Hodge structure (building off Saito's theory of mixed Hodge modules \cite{Saito}) on the equivariant cohomology of the vanishing cycle complex, and then \cite{DMdim1} and \cite{BenPos} built on these ideas to rigorously define a quantum integration map $\s{I}_t$.

We give a brief sketch of this \textbf{integration map} \begin{align*}
    \s{I}_t:H_{\reg}(Q,W)\rar \bb{C}_t[N^{\oplus}]
\end{align*} 
essentially as in \cite[\S 3.3]{BenPos}.  We then use this to compute the integration in the simplest cases.  We note that by the definitions of the Poisson structures in \eqref{regPois} and \eqref{qPois}, it is clear that $\s{I}_t$ being a map of algebras implies it is also a map of Poisson algebras, thus also giving maps of Lie algebras as in Remark \ref{PvC}.

Recall that $\s{M}$ is the moduli stack of objects in $\rep(Q,W)\coloneqq \bmod \C Q/I_W$.  Let $\s{M}^{\circ}$ be the moduli stack of objects in $\rep(Q,0)$.  Given an arrow $a\in Q_1$, let $t(a),h(a)\in Q_0$ denote the tail and head of $a$ respectively.  For any $i\in Q_0$ and $d\in N^{\oplus}$, let $d_i$ denote the corresponding component of $d$.  Denote
\begin{align*}
    \wt{\s{M}}_d^{\circ}\coloneqq \prod_{a\in Q_1}\Hom_{\bb{C}}(\C^{d_{t(a)}},\C^{d_{h(a)}}),
\end{align*}
and 
\begin{align*}
    \GL_d\coloneqq \prod_{i\in Q_0} \GL_{d_i}(\bb{C}).
\end{align*}
Then $\s{M}^{\circ}=\bigsqcup_{d\in N^{\oplus}} \s{M}_d^{\circ}$ where $\s{M}_d^{\circ}$ is the stack-theoretic quotient
\begin{align}\label{M-circ-stack}
    \s{M}_d^{\circ}=\wt{\s{M}}_d^{\circ}/\GL_d,
\end{align} 
where the action by $\GL_d$ is the one induced by the conjugation action of $\GL_{d_i}(\bb{C})$ on $\C^{d_i}$ for each $i\in Q_0$.

Viewing elements of $\wt{\s{M}}_d^{\circ}$ as modules over the path-algebra $\C Q$, we see that multiplication by $W$ gives an endomorphism of $\wt{\s{M}}_d^{\circ}$.  Since the trace is invariant under the action of $\GL_d$, we obtain a function $$\Tr(W):\s{M}^{\circ}\rar \bb{C},$$ 
the critical locus of which recovers $\s{M}$:
\begin{align*}
    \s{M}=\crit(\Tr(W))\subset \s{M}^{\circ}.
\end{align*}

Let $Y$ be a smooth complex variety and let $f:Y\rar \bb{C}$ be a regular function on $Y$.  The corresponding vanishing cycle functor $\varphi_f$ is defined as follows (following \cite[\S 3.1]{BenPos}, also cf. \cite[\S 7.2]{KScoha}).   Let $Y_0\coloneqq f^{-1}(0)$, and let $Y_{\leq 0}\coloneqq f^{-1}(\bb{R}_{\leq 0})$.  For a sheaf $\s{F}$ on $Y$ and $U$ an analytic open subset of $Y$, define
\begin{align*}
    \Gamma_{X_{\leq 0}}\s{F}(U)\coloneqq \ker \left(\s{F}(U)\rar \s{F}(U\setminus (U\cap X_{\leq 0})\right).
\end{align*}
Then $\varphi_f:=(R\Gamma_{X\leq 0} \s{F})[1]|_{X_0}$.
%[This should have been here before, somehow accidentally got left out.]

The stacks $\s{M}_d^{\circ}$ for $d\in N^{\oplus}$ are not quite smooth complex varieties, but each is a quotient of a smooth complex variety by the action of an algebraic group, cf. \eqref{M-circ-stack}.  One can thus extend the definition of $\varphi_f$ to regular functions $f$ on $\s{M}_d^{\circ}$ using an equivariant version of the vanishing cycle construction as in \cite[\S 2.2]{DM}.

Let $\u{\Q}_d$ denote the constant sheaf on $\s{M}_d^{\circ}$.  For each $u\in \bb{C}^*$, we can define $\varphi_{\Tr(W)/u}\u{\Q}_d$.  Now consider $[X\rar \s{M}_d]\in H_{\reg}(Q,W)$. Composing with the inclusion $\s{M}_d\subset \s{M}_d^{\circ}$, we can consider the pullback
\begin{align*}
    \varphi^X_{\Tr(W)/u} \u{\Q}_d\coloneqq  (X\rar \s{M}_d^{\circ})^*\varphi_{\Tr(W)/u}\u{\Q}_d.
\end{align*}
This sheaf $\varphi^X_{\Tr(W)/u} \u{\Q}_d$ on $X$ in fact has the structure of a mixed Hodge module on $X$, and so the compactly supported cohomology $H^*_c(X,\varphi^X_{\Tr(W)/u} \u{\Q}_d)$ has a cohomologically graded rational mixed Hodge structure.  Recall here that a rational mixed Hodge structure is a finite-dimensional vector space $V$ over $\bb{Q}$, plus the data of an ascending filtration $W_*$ of $V$ (the weight filtration) and a descending filtration $F^*$ of $V\otimes_{\bb{Q}} \bb{C}$ (the Hodge filtration) such that the filtration induced by $F^*$ on
\begin{align*}
    \Gr^W_n(V):=W_n \otimes_{\bb{Q}} \bb{C} / W_{n-1} \otimes_{\bb{Q}} \bb{C}
\end{align*}
determines a pure Hodge structure of weight $n$.  By a cohomologically graded rational mixed Hodge structure on $H^*_c(X,\varphi^X_{\Tr(W)/u} \u{\Q}_d)$, we mean a rational mixed Hodge structure on $H^i_c(X,\varphi^X_{\Tr(W)/u} \u{\Q}_d)$ for each $i\in \bb{Z}$.

Let us abbreviate $H^*_c(X,\varphi^X_{\Tr(W)/u} \u{\Q}_d)$ as simply $H^*_d$. Up to isomorphism, the sheaf $\varphi^X_{\Tr(W)/u}\u{\Q}_d$ is independent of $u$.  However, there may be non-trivial monodromy $\mu$ on $H^*_d$ as $u$ travels around the origin in $\bb{C}$.  This $\mu$ is quasi-unipotent, i.e., the eigenvalues are roots of unity.  Let $\Gr^W_n(H^i_d)_1$ denote the generalized eigenspace for the possible eigenvalue $1$ of $\mu$, and let $\Gr^W_n(H^i_d)_{\neq 1}$ denote the direct sum of the generalized eigenspaces for all eigenvalues of $\mu$ other than $1$.  Finally, the quantum integration map $\s{I}_t$ is defined by taking the Serre polynomial (cf. \cite[\S 3.1.3]{DL}, also \cite[p. 69]{KSmotivic}) defined as follows:
\begin{align}\label{It}
\s{I}_t([X\rar \s{M}_d])\coloneqq t^{\chi(d,d)}z^d\sum_{i,n \in \bb{Z}} (-1)^{i} &\left(\dim(\Gr^W_n(H^i_d)_1)(-t)^{n} + \dim(\Gr^W_n(H^i_d)_{\neq 1})(-t)^{n+1}\right),
\end{align}
where $\chi$ is defined as in \eqref{Euler}.

The map of \eqref{It} above is essentially the same as that of \cite[(18)]{BenPos}, although the two look somewhat different.  The $t^{\chi(d,d)}$-factor in our \eqref{It} is simply to account for the twisting of the monoidal structure in \cite[(16)]{BenPos}.  The extra factor of $(-t)$ on the $\dim(\Gr^W_n(H^i_d)_{\neq 1})(-t)^{n+1}$-term in our \eqref{It} is needed because \cite{BenPos} actually works with the category of monodromic mixed Hodge modules, a difference which results in a shift for part of the weight filtration.  See \cite[Prop. 2.5]{DavJacobi} for details on the relationship between these two perspectives.  \cite[Prop. 3.13]{BenPos} thus yields the following:

%roughly meaning that $\varphi^X_{\Tr(W)/u}$ is viewed as a mixed Hodge module on $X\times \bb{C}^*$, modulo those which extend to a constant mixed Hodge module over $\bb{A}^1$.  This results in a shift of the weights for the non-unipotent parts of  a result of a difference in the weight filtrations when viewing   [some grading issue when working in MMHM, I need to ask Ben about this still.]  

%as a mixed Hodge    In this viewpoint, one would define $\varphi^X_{\Tr(W)/u} \u{\Q}_d$ as $(X\times \bb{C}^*\rar \s{M}_d^{\circ}\times \bb{C}^*)^*\varphi_{\Tr(W)/u}\u{\Q}_d$, viewed as a constructible stack over $\bb{C}^*$ (here the parameter $u$ is the coordinate on the $\bb{C}^*$-factor).  Then in place of $H^*_c(X,\varphi^X_{\Tr(W)/u} \u{\Q}_d)$, one considers the $j_!$ of the constant sheaf $\u{\Q}$ on $\varphi^X_{\Tr(W)/u}$ where $j$ is the projection to the $\bb{C}^*$-factor.  This gives a mixed Hodge module, and this itself   \

\begin{prop}[\cite{BenPos}, Prop. 3.13]
$\s{I}_t:H_{\reg}(Q,W)\rar \bb{C}_t[N^{\oplus}]$ is a homomorphism of $\C_{\reg}(t)$-algebras.
\end{prop}

This construction simplifies quite a bit for $[f:X\rar \s{M}_d]$ with $\Tr(W)|_{\s{M}_d^{\circ}}=0$ and $X$ a smooth projective variety.  In this case, $\varphi_{\Tr(W)}=\id$, and so $\Gr^W_n(H^i_d)$ equals $H_c^{n}(X,\u\Q)$ if $n=i$ and vanishes otherwise.  Recalling the definition of $\Upsilon$ from \eqref{Ups}, we thus recover the following:
\begin{prop}\label{I-Ups}
If $[f:X\rar \s{M}_d\subset \s{M}]\in H_{\reg}(Q,W)$ and $\Tr(W)|_{\s{M}_d^{\circ}}=0$, then
\begin{align}\label{ItUps}
    \s{I}_t([f:X\rar \s{M}])=\Upsilon(X)t^{\chi(d,d)}z^d.
\end{align}
\end{prop}
In particular, if $W=0$ (e.g., for $Q$ acyclic), $\s{I}_t$ equals the quantum integration map of \cite[\S 6]{joyce2007configurations}.

\begin{eg}\label{IkappaEx}
Recall $\kappa_{ki}\coloneqq [\Spec \bb{C} \rar p_{S_i^{\oplus k}}=\s{M}_{ke_i}]$.  We have $W|_{\s{M}^{\circ}_{ke_i}}=0$, $\Upsilon(\kappa_{ki})=1$ (the Poincar\'e polynomial of a point), and $\chi(ke_i,ke_i)=k^2$.  Hence, $\s{I}_t(\kappa_{ki})=t^{k^2}z^{ke_i}$.  In particular,
\begin{align}\label{Ikappa}
    \s{I}_t(\kappa_i)=tz^{e_i}.
\end{align}
As a check, one can use \eqref{kappaik} to confirm that $\s{I}_t(\kappa_i^k)=\s{I}_t(\kappa_i)^k$.
\end{eg}

Composing $\s{I}_t$ with $\pi_{t\mapsto 1}$ induces the classical integration map:
\begin{align*}
    \s{I}\coloneqq \pi_{t\mapsto 1} \circ \s{I}_t : H_{\reg}(Q,W) \rar \bb{C}[N^{\oplus}].
\end{align*}
The classical integration maps of \cite[\S 7]{JS} and \cite[Thm. 11.1]{Bridge} are always (even for nonzero $W$) given by the $t\mapsto 1$ limit of \eqref{ItUps}, i.e., by taking Euler characteristics.  Note that \eqref{Ikappa} is sufficient to completely determine the restrictions of $\s{I}_t$ and $\s{I}$ to the composition algebra $\s{C}(Q,W)$ in which all our computations will lie.  Since $\s{I}$ agrees with the classical integration maps of \cite[\S 7]{JS} and \cite[Thm. 11.1]{Bridge} on the generators $\kappa_i$, the maps necessarily agree on all of $\s{C}(Q,W)$.

\section{Scattering diagrams from Hall algebras}\label{ScatterSection}

\subsection{Background on scattering diagrams}\label{Scattering}

Here we review the basic definitions and properties of scattering diagrams from the perspective useful for understanding the Hall algebra scattering diagrams of \cite{Bridge}. 

Let $\Lambda$ denote a finite-rank lattice equipped with a $\bb{Z}$-valued skew-symmetric form $\{\cdot,\cdot\}$.  Let $\Lambda^{\vee}\coloneqq \Hom(\Lambda,\bb{Z})$ be the dual lattice, and let $\langle\cdot,\cdot\rangle:\Lambda\oplus \Lambda^{\vee}\rar \bb{Z}$ denote the dual pairing.  We have a map
\begin{align}\label{pistar}
p^*: \Lambda&\rar \Lambda^{\vee} \nonumber\\
n&\mapsto \{\cdot,n\}.
\end{align}
Fix a strictly convex rational polyhedral cone $\sigma_{\Lambda^{\oplus}}\subset \Lambda_{\bb{R}}$.  Let $\Lambda^{\oplus}\coloneqq \sigma_{\Lambda^{\oplus}}\cap \Lambda$, and let $\Lambda^+\coloneqq \Lambda^{\oplus}\setminus \{0\}$.

Let $\f{g}\coloneqq \bigoplus_{n\in \Lambda^+} \f{g}_n$ be a Lie algebra graded by $\Lambda^+$, meaning that $[\f{g}_{n_1},\f{g}_{n_2}] \subseteq \f{g}_{n_1+n_2}$.  We say that $\f{g}$ is \textbf{skew-symmetric} with respect to $\{\cdot,\cdot\}$ if
\begin{align}\label{skewCondition}
[\f{g}_{n_1},\f{g}_{n_2}]=0 \mbox{~whenever~}\{n_1,n_2\}=0.  
\end{align}

For each $k\in \bb{Z}_{\geq 1}$, let $$k\Lambda^+\coloneqq \{n_1+\ldots+n_k\in \Lambda^+|n_i\in \Lambda^+ \mbox{ for each } i=1,\ldots,k\}.$$  Let $\f{g}^{\geq k}\coloneqq \bigoplus_{n\in k\Lambda^+} \f{g}_n$.  Note that $\f{g}^{\geq k}$ is a Lie subalgebra of $\f{g}$.  Let $\f{g}_k$ denote the nilpotent Lie algebra $\f{g}/\f{g}^{\geq k}$, and let $\wh{\f{g}}\coloneqq \varprojlim \f{g}_k$.  We have corresponding Lie groups $G\coloneqq \exp \f{g}$, $G_k\coloneqq \exp \f{g}_k$, and $\wh{G}\coloneqq \exp \wh{\f{g}} = \varprojlim G_k$.

For each $n \in \Lambda^+$, we have a Lie subalgebra $\f{g}_{n}^{\parallel}\coloneqq  \prod_{k \in \bb{Z}_{\geq 1}} \f{g}_{kn} \subset \wh{\f{g}}$.   We say that $\f{g}$ has \textbf{Abelian walls} if each $\f{g}_{n}^{\parallel}$ is Abelian.  In particular, $\f{g}$ has Abelian walls whenever $\f{g}$ is skew-symmetric.  Let $G_{n}^{\parallel}\coloneqq \exp(\f{g}_n^{\parallel}) \subset \wh{G}$.

The Abelian walls condition is usually assumed to hold when working with scattering diagrams, but when defining Hall algebra scattering diagrams, one needs a slight generalization as in \cite[\S 2]{Bridge}.

\begin{dfn}\label{WallDfn}
A \textbf{wall} in $\Lambda^{\vee}_{\bb{R}}$ over $\wh{\f{g}}$ is data of the form $(\f{d},g_{\f{d}})$, where:
\begin{itemize}
\item $g_{\f{d}}\in \f{g}_{n_{\f{d}}}^{\parallel}$ for some primitive $n_{\f{d}}\in \Lambda^+$.  The element $-p^*(n_{\f{d}})$ is called the \textbf{direction} of the wall.  We call $g_{\f{d}}$ the \textbf{scattering function} associated to the wall.
\item $\f{d}$ is a closed, convex (but not necessarily strictly convex), rational-polyhedral, codimension-one affine cone in $\Lambda^{\vee}_{\bb{R}}$, parallel to $n_{\f{d}}^{\perp}$.  We call $\f{d}$ the \textbf{support} of the wall.
\end{itemize}

A \textbf{scattering diagram} $\f{D}$ over $\wh{\f{g}}$ is a set of walls in $\Lambda^{\vee}_{\bb{R}}$ over $\wh{\f{g}}$ such that for each $k >0$, there are only finitely many $(\f{d},g_{\f{d}})\in \f{D}$ with $g_{\f{d}}$ not projecting to $0$ in $\f{g}_k$.  If $(\f{d}_1,g_{\f{d}_1})$ and $(\f{d}_2,\f{g}_{\f{d}_2})$ are two walls of $\f{D}$, and if $\codim_{\Lambda^{\vee}_{\bb{R}}}(\f{d}_1\cap \f{d}_2)=1$, then we require that $[g_{\f{d}_1},g_{\f{d}_2}]=0$ (note that this is automatic for Abelian walls).

A wall with direction $-v$ is called \textbf{incoming} if it contains $v$.  Otherwise, the wall is called \textbf{outgoing}.
\end{dfn}

We will sometimes denote a wall $(\f{d},g_{\f{d}})$ by just $\f{d}$.  Denote $\Supp(\f{D})\coloneqq  \bigcup_{\f{d}\in \f{D}} \f{d}$, and \begin{align*}
\Joints(\f{D})\coloneqq  \bigcup_{\f{d}\in \f{D}} \partial \f{d} \cup \bigcup_{\substack{\f{d}_1,\f{d}_2\in \f{D}\\
		               \dim (\f{d}_1\cap \f{d}_2) = \rank(\Lambda)-2}} \f{d}_1\cap \f{d}_2. 
\end{align*}

Note that for each $k>0$, a scattering diagram $\f{D}$ over $\wh{\f{g}}$ induces a finite scattering diagram $\f{D}^k$ over $\f{g}_k$ with walls corresponding to the $\f{d}\in \f{D}$ for which the projection of $g_{\f{d}}$ to $\f{g}_k$ is nonzero.

Consider a smooth immersion $\gamma:[0,1]\rar \Lambda^{\vee}_{\bb{R}}\setminus \Joints(\f{D})$ with endpoints not in $\Supp(\f{D})$ which is transverse to each wall of $\f{D}$ it crosses.  Let $(\f{d}_i,g_{\f{d}_i})$, $i=1,\ldots, s$, denote the walls of $\f{D}^{k}$ crossed by $\gamma$, and say they are crossed at times $0<t_1\leq \ldots \leq t_s<1$, respectively.\footnote{If $t_i=t_{i+1}$, then the corresponding elements $g_{\f{d}_i},g_{\f{d}_{i+1}}\in \f{g}_n^{\parallel}$ must commute, and so the ordering of the corresponding walls does not affect $\Phi_{\gamma,\f{D}}^k$.}  Define 
\begin{align}\label{WallCross}
\Phi_{\f{d}_i}\coloneqq \exp(g_{\f{d}_i})^{\sign \langle n_{\f{d}_i},-\gamma'(t_i)\rangle} \in G_k.
\end{align}
Let $\Phi_{\gamma,\f{D}}^k\coloneqq \Phi_{\f{d}_s} \cdots \Phi_{\f{d}_1}\in G_k$, and define the \textbf{path-ordered product}:
\begin{align*}
\Phi_{\gamma,\f{D}}\coloneqq  \varprojlim_k \Phi_{\gamma,\f{D}}^k \in \wh{G}.
\end{align*}

\begin{dfn}
Two scattering diagrams $\f{D}$ and $\f{D}'$ are \textbf{equivalent} if $\Phi_{\gamma,\f{D}} = \Phi_{\gamma,\f{D}'}$ for each smooth immersion $\gamma$ as above.  $\f{D}$ is \textbf{consistent} if each $\Phi_{\gamma,\f{D}}$ depends only on the endpoints of $\gamma$.
\end{dfn}

We say that $x\in \Lambda_{\bb{R}}^{\vee}$ is \textbf{general} if it is contained in at most one hyperplane of the form $n^{\perp}$ for $n\in \Lambda$.  For $\f{D}$ a scattering diagram over $\wh{\f{g}}$ and $x\in \Lambda_{\bb{R}}^{\vee}$ general, denote $$g_{x,\f{D}}\coloneqq \sum_{\f{d}\ni x} g_{\f{d}}\in \wh{\f{g}},$$ where the sum is over all walls $(\f{d},g_{\f{d}})\in \f{D}$ with $\f{d}\ni x$.  One easily sees the following standard fact  (cf. \cite[Lem. 1.9]{GHKK}):
\begin{lem}\label{gxEquiv}
Two scattering diagrams $\f{D}$ and $\f{D}'$ over $\wh{\f{g}}$ are equivalent if and only if $g_{x,\f{D}}=g_{x,\f{D}'}$ for all general $x\in \Lambda_{\bb{R}}^{\vee}$.
\end{lem}

\begin{eg}\label{eq}
~

\begin{enumerate}
    \item For $\f{D}$ a scattering diagram, consider a set of walls $\{(\f{d},g_i)\in \f{g}_{n_{\f{d}}}^{\parallel})\in \f{D}| i\in S\}$, where $S$ is some countable index set and $n_{\f{d}}$ and $\f{d}$ are independent of $i$.  Then replacing this set of walls with a single wall $(\f{d},\sum_{i\in S} g_{i})$ produces an equivalent scattering diagram.
    \item Replacing a wall $(\f{d},g_{\f{d}})\in \f{D}$ with a pair of walls $(\f{d}_i,g_{\f{d}})$, $i=1,2$, such that $\f{d}_1\cup \f{d}_2=\f{d}$ and $\codim_{\Lambda_{\bb{R}}^{\vee}}(\f{d}_1\cap \f{d}_2)=2$ produces an equivalent scattering diagram.
\end{enumerate}
\end{eg}

The following theorem is fundamental to the study of scattering diagrams.  The $2$-dimensional version was first proved in \cite{KS}, and this was generalized to higher dimensions in \cite[\S 3]{GS11} for scattering diagrams over the module of log derivations. The higher-dimensional version for scattering diagrams over skew-symmetric Lie algebras follows from \cite[Prop. 3.2.6, 3.3.2]{WCS} (cf. \cite[Thm. 1.21]{GHKK} for a review of this argument from our viewpoint).  As pointed out to us by Lang Mou, this result had not previously been proven in the presence of non-Abelian walls.

\begin{thm}\label{KSGS}
Let $\f{g}$ be a $\Lambda^+$-graded Lie algebra, and let $\f{D}_{\In}$ be a finite scattering diagram over $\wh{\f{g}}$ whose walls are of the form $(n_i^{\perp},g_i)$ for various primitive $n_i\in N^+$.  If $\f{g}$ has Abelian walls, then there is a unique-up-to-equivalence scattering diagram $\f{D}$ such that $\f{D}$ is consistent, $\f{D} \supset \f{D}_{\In}$, and $\f{D}\setminus \f{D}_{\In}$ consists only of outgoing walls.  Even if $\f{g}$ does not have Abelian walls, if there exists a consistent scattering diagram $\f{D}\supset \f{D}_{\In}$ such that $\f{D}\setminus \f{D}_{\In}$ consists only of outgoing walls as above, then this $\f{D}$ is the unique such scattering diagram, up to equivalence.
\end{thm}
We note that an earlier version of this paper claimed existence more generally, but we have since realized that proving the consistency of the scattering diagram $\f{D}_k^{\infty}$ in \S \ref{TropScatProof} requires the Abelian walls condition, and so our argument was flawed.  Fortunately, the existence of the Hall algebra scattering diagram is already given by \cite[Theorem 6.5]{Bridge}, restated below as Theorem \ref{thm:Hallscattering}.  

\begin{proof}
As noted above, the only new statement is the uniqueness statement in the case of non-Abelian walls.  We prove this using an argument inspired by that \cite[Lem. C.7]{GHKK}.  Let $\f{D},\f{D}'$ be two consistent scattering diagrams over $\wh{\f{g}}$ with incoming walls $\f{D}_{\In}$ as in the statement of the theorem.  We shall prove by induction on $k$ that $\f{D}^k$ and $(\f{D}')^k$ are equivalent over $\f{g}_k$ for each $k$, and then the equivalence of $\f{D}$ and $\f{D}'$ follows.  Note that $\f{D}^1$ and $(\f{D}')^1$ are both equivalent to the trivial scattering diagram, hence to each other.

Now suppose that $\f{D}^k$ and $(\f{D}')^k$ are equivalent over $\f{g}_k$.  Let $\f{D}''$ be a scattering diagram over $\f{g}_{k+1}$ such that $$g_{x,\f{D}''}=g_{x,\f{D}^{k+1}}-g_{x,(\f{D}')^{k+1}}$$ for each general $x\in \Lambda_{\bb{R}}^{\vee}$. 
Since $\f{D}^k$ and $(\f{D}')^k$ are equivalent over $\f{g}_k$, we must have $g_{x,\f{D}''}\in \f{g}^{\geq k}\setminus \f{g}^{\geq k-1}$, hence $g_{x,\f{D}''}$ is central in $\f{g}_{k+1}$.  Hence, $(\f{D}')^k \cup \f{D}''$ is a well-defined scattering diagram over $\f{g}_{k+1}$, and by Lemma \ref{gxEquiv} it is equivalent to $\f{D}^{k+1}$.  Our goal now is to show that $\f{D}''$ is equivalent to the trivial scattering diagram.

Since both $\f{D}$ and $\f{D}'$ were assumed to be consistent, and the scattering functions of $\f{D}''$ are all central in $\f{g}_{k+1}$, $\f{D}''$ must also be consistent (over $\f{g}_{k+1}$).  Furthermore, this consistency plus centrality of the scattering functions implies that, up to equivalence, the support of every wall of $\f{D}''$ is an entire affine hyperplane in $\Lambda_{\bb{R}}^{\vee}$.  But then all walls of $\f{D}''$ (up to equivalence) are incoming, and since the incoming walls of $\f{D}$ and $\f{D}'$ are the same, this implies that $\f{D}''$ is equivalent to the trivial scattering diagram over $\f{g}_{k+1}$, as desired.
\end{proof}

A scattering diagram playing the role of $\f{D}_{\In}$ in Theorem \ref{KSGS} will be referred to as an \textbf{initial scattering diagram}.  The consistent scattering diagram $\f{D}$ (up to equivalence) with incoming walls $\f{D}_{\In}$ as in the theorem will be denoted $\scat(\f{D}_{\In})$.

\begin{eg} \label{ex:firstdiag}
Consider $\Lambda = \Z^2$.  Equip $\Lambda$ with the skew-symmetric form $\{\cdot,\cdot\}$ represented by 
$\left( \begin{array}{c c}
0 & -1 \\ 1 & 0
\end{array}
\right)$, and  consider the quantum torus algebra $\bb{C}_t[\Lambda]$ as in \S \ref{qtor}.  Take $\f{g}$ to be the Lie subalgebra (with respect to Poisson bracket) with basis $\{z^n:n\in \Lambda^+\}$.  Let $$\f{D}_{\In}\coloneqq \{(e_1^{\perp},-\Li(-z^{e_1};t)),(e_2^{\perp},-\Li(-z^{e_2};t))\},$$
where $\Li(x,t)$ denotes the quantum dilogarithm as in \eqref{eq:li} below.  Then $\f{D}\coloneqq \scat(\f{D}_{\In})$ is obtained by adding a single outgoing wall $(\bb{R}_{\geq 0} (1,-1), -\Li(-z^{(1,1)};t))$, cf. Figure \ref{a2q}.

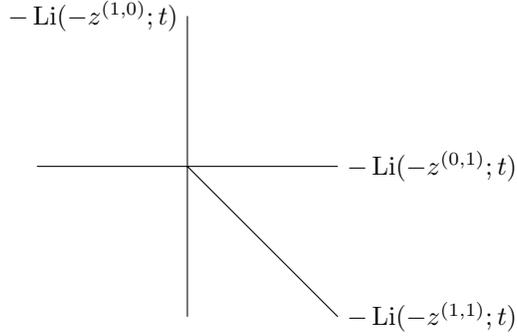
\begin{figure}[htb]
\centering
\begin{tikzpicture}
\draw
(-2,0) -- (2,0) node[right] {$-\Li(-z^{(0,1)}; t)$}
(0,-2) -- (0,2) node[left] {$-\Li(-z^{(1,0)}; t)$}
(0,0) -- (2,-2) node[right] {$-\Li(-z^{(1,1)}; t)$};
\end{tikzpicture}
\caption{The quantum $A_2$ scattering diagram.  \label{a2q}}
\end{figure}
 
The consistency of this scattering diagram is equivalent to a version of the quantum pentagon identity of \cite{FK}.  The classical limit is essentially the $\ell_1=\ell_2=1$ case of \cite[Ex. 1.6]{GPS} (with some small changes in sign conventions).  We will see in Example \ref{ex:a2diag} that this is the scattering diagram obtained when applying the quantum integration map to the Hall algebra scattering diagram associated to the $A_2$-quiver.
\end{eg}

\subsection{Hall algebra scattering diagrams}\label{HallScat}

\subsubsection{Setup for Hall algebra scattering diagrams and their variants}\label{HallScatSetup}

We now take $\Lambda=N$,  $\Lambda^{\oplus}=N^{\oplus}$, and $\{\cdot,\cdot\}=B$.  Recall that $H(Q,W)$ admits a grading by $N^{\oplus}$ as in \eqref{Hgrade}. In particular, we can write $H_{\reg}(Q,W)=H_{\reg}(Q,W)_0\oplus H_{\reg}(Q,W)_{>0}$ for $H_{\reg}(Q,W)_{>0}\coloneqq \bigoplus_{d\in N^+} H_{\reg}(Q,W)_d$.  Let $\f{g}^{\Hall}\coloneqq (t-t^{-1})^{-1}\cdot H_{\reg}(Q,W)_{>0}$, viewed as a Lie algebra using the commutator bracket as in Remark \ref{PvC}.

The Lie algebra $\f{g}^{\Hall}$ typically is not skew-symmetric and does not have Abelian walls.  To get around this issue, let $\f{i}^{\Skew}$ denote the Lie ideal of $\f{g}^{\Hall}$ generated by the commutators we wish to vanish, i.e.,
\begin{align*}
    \f{i}^{\Skew}\coloneqq \left\langle [\f{g}^{\Hall}_{d_1},\f{g}^{\Hall}_{d_2}] : d_1,d_2\in N^+, \{d_1,d_2\}=0\right\rangle.
\end{align*}
Here, for $S$ a subset of $\f{g}_{\reg}$, $\langle S \rangle$ denotes the Lie ideal generated by $S$, i.e., the intersection of all Lie ideals of $\f{g}^{\Hall}$ which contain $S$. 
Then for any Lie ideal $\f{i}$ which contains $\f{i}^{\Skew}$, we define
\begin{align*}
    \f{g}^{\f{i}}\coloneqq \f{g}^{\Hall}/\f{i}.
\end{align*}
Note that for any Lie algebra ideal $\f{i}$ of $\f{g}^{\Hall}$, $\f{g}^{\Hall}/\f{i}$ is skew-symmetric if and only if $\f{i}\supset \f{i}^{\Skew}$.  Since the commutator bracket on the quantum torus algebra makes it into a skew-symmetric Lie algebra, we in particular have
\begin{align*}
    \ker(\s{I}_t)\supset \f{i}^{\Skew}.
\end{align*}
The resulting Lie algebra $\f{g}^q\coloneqq \f{g}^{\ker(\s{I}_t)}$ is just the quantum torus algebra $(t-t^{-1})^{-1}\cdot \bb{C}_t[N^{\oplus}]$ with its commutator bracket as in \eqref{qPois}.  Similarly, $\ker(\s{I})\supset \f{i}^{\Skew}$, and $\f{g}^{\scl}\coloneqq \f{g}^{\ker(\s{I})}$ is just $\bb{C}[N^{\oplus}]$ together with its Poisson bracket as in \eqref{Pois}.  In general, let $\s{I}^{\f{i}}:\f{g}^{\Hall}\rar \f{g}^{\f{i}}$ denote the projection.

For $\f{g}$ equal to $\f{g}^{\Hall}$, $\f{g}^{\f{i}}$, $\f{g}^q$, or $\f{g}^{\scl}$, we denote the corresponding Lie group $G$ by $G^{\Hall}$, $G^{\f{i}}$, $G^q$, or $G^{\scl}$, respectively.  The notation for the associated completions and scattering diagrams will be similarly obvious except for sometimes using ``$\Hall$'' instead of ``$\reg$.''\footnote{Note that we could define the Hall algebra scattering diagram using the full Hall algebra $H(Q,W)$ in place of $H_{\reg}(Q,W)$ (as is done in \cite{Bridge}), or alternatively using just the composition algebra $\s{C}(Q,W)$.  The advantage of using $H_{\reg}(Q,W)$ or $\s{C}(Q,W)$ instead of $H(Q,W)$ is just for convenience when we talk about applying integration maps.}

\subsubsection{The Hall algebra scattering diagram}\label{HallScatsubsub}

\begin{dfn}\label{semistableDef}
Given $\stab \in M_{\R}$, an object $E \in \rep(Q,W)$ is said to be \textbf{$\stab$-semistable} if 
\begin{itemize}
	\item $\stab(E) = 0$, 
	\item every subobject $B \subset E$ satisfies $\stab(B) \leq 0$.  If, furthermore, this inequality is strict, then we say that $E$ is $\stab$-stable.
\end{itemize}
The notion of semistability given above is due to \cite{King}.
Let $\s{M}_{\sst}(\stab)\subset \s{M}$ denote the substack of $\s{M}$ representing the $\stab$-semistable objects, and let $1_{\sst}(\stab):=[\s{M}_{\sst}(\stab)\subset \s{M}] \in \wh{G}_{\Hall}$.  
\end{dfn}

The scattering diagram defined in the following theorem of Bridgeland is what one calls the \textbf{Hall algebra scattering diagram}.
\begin{thm} \cite[Theorem 6.5]{Bridge} \label{thm:Hallscattering}
There exists a consistent scattering diagram $\frakD^{\Hall}$ in $M_{\R}$ over $\f{g}^{\Hall}$ such that:

\begin{enumerate}
	\item The support $\Supp(\f{D}^{\Hall})$ consists of those $\stab \in M_{\R}$ for which there exist $\stab$-semistable objects in $\rep(Q,I)$; 
	\item For $\stab \subset \Supp(\frakD^{\Hall})\setminus \Joints(\frakD^{\Hall})$, there is a unique wall $(\f{d},g_{\f{d}})\in \f{D}^{\Hall}$ for which $\f{d}\ni \stab$.  For this wall, we have 
$\exp(g_{\f{d}}) = 1_{\sst} (\stab) \in \hat{G}_{\Hall}$.
\end{enumerate}
\end{thm}
\begin{rmk}\label{general}
We say $\stab\in M_{\bb{R}}$ is \textbf{general} if it is not in the intersection of two distinct hyperplanes of the form $n^{\perp}$ for $n\in N\setminus \{0\}$.  Since the joints of $\f{D}^{\Hall}$ are codimension $2$ subsets of $M_{\bb{R}}$ and have rational slope, Theorem \ref{thm:Hallscattering} gives the scattering functions of $\f{D}^{\Hall}$ at all general points $\stab\in M_{\bb{R}}$.  Alternatively, we could use a more refined notion of general.  Call $\stab\in M_{\bb{R}}$  \textbf{special} if at least one of the following holds:
\begin{itemize}
    \item There exists a pair of $\stab$-semistable objects with non-parallel dimension vectors;
    \item Some $E\in \rep(Q,W)$ is $\stab$-semistable, but for $0<\epsilon\ll 1$, $E$ is either not $(\stab+\epsilon p^*(\dim(E)))$-semistable or not $(\stab-\epsilon p^*(\dim(E)))$-semistable.
\end{itemize}
The former condition accounts for joints where two walls of different slopes intersect, while the latter accounts for intersections of walls with the same slope.  That is, $\stab\in \Joints(\f{D}^{\Hall})$ if and only if $\stab$ is special.  Theorem \ref{MainIntro} will still hold and will be slightly stronger if we define general to mean not special.
\end{rmk}

Note that we obtain new scattering diagrams $\f{D}^{\f{i}}$, $\f{D}^q$, and $\f{D}^{\scl}$ over $\f{g}^{\f{i}}$, $\f{g}^q$, and $\f{g}^{\scl}$, respectively, by applying $\s{I}^{\f{i}}$, $\s{I}^q$, or $\s{I}^{\scl}$ to $\f{D}^{\Hall}$.  The scattering diagram $\f{D}^{\scl}$ is what Bridgeland calls the \textbf{stability scattering diagram}.  We call $\f{D}^q$ the \textbf{quantum stability scattering diagram}.  

\begin{eg} \label{ex:a2diag}
Let us consider the $A_2$ quiver $1  \rightarrow 2$  with $W=0$. 
The corresponding matrix $B$ is $\left(
\begin{array}{c c}
0 & -1 \\ 1 & 0
\end{array}
\right)$ as in Example \ref{ex:firstdiag}.  Let us explicitly describe the Hall algebra scattering diagram $\f{D}^{\Hall}$ from Theorem \ref{thm:Hallscattering} in this case.  
Note that there are 3 indecomposable representations of $A_2$ up to isomorphism: $\C \rightarrow 0$, $0 \rightarrow \C$, and $\C \rightarrow \C$ (the last map being nonzero).  Consider $\frakd=(1,0)^{\perp}$. 
For any point $\stab \in \frakd$, one can see that the representations $(\C \rightarrow 0)^{\oplus k}$ are $\stab$-semistable for any positive integer $k$, and we find $1_{\sst}(\stab)=\sum_{k\geq 0}(\bb{C}\rightarrow 0)^{\oplus k}$.  We similarly compute that for $\stab\in (0,1)^{\perp}$, $1_{\sst}(\stab)=\sum_{k\geq 0}(0\rightarrow \bb{C})^{\oplus k}$, and for $\stab\in \bb{R}_{\geq 0}(1,-1)$, we have $1_{\sst}(\stab)=\sum_{k\geq 0}(\C\rightarrow \C)^{\oplus k}$.
Note that $(\C \rightarrow \C)$ contains $(0 \rightarrow \C)$ as a subrepresentation, and so $(\C \rightarrow \C)^{\oplus k} $ is not $ (-\alpha,\alpha)$-semistable for $\alpha \in \bb{R}_{>0}$.  There are no other $\stab$-semistable representations for any $\stab$ in this example, so the Hall algebra scattering diagram is as in Figure \ref{a2}.  Note that $\f{D}^{q}$, obtained from applying the quantum integration map $\s{I}$ to the scattering functions of $\f{D}^{\Hall}$ (cf. \S \ref{HallInitialSection} for such computations) yields the consistent scattering diagram of Example \ref{ex:firstdiag}.

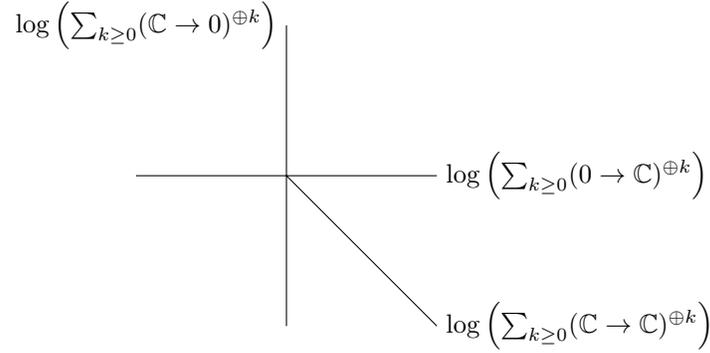
\begin{figure}[htb]
\centering
\begin{tikzpicture}
\draw
(-2,0) -- (2,0) node[right] {$\log \left(\sum_{k\geq 0} (0 \rightarrow \C)^{\oplus k}\right)$}
(0,-2) -- (0,2) node[left] {$\log \left(\sum_{k\geq 0} (\C \rightarrow 0)^{\oplus k}\right)$}
(0,0) -- (2,-2) node[right] {$\log \left(\sum_{k\geq 0} (\C \rightarrow \C)^{\oplus k}\right)$};
\end{tikzpicture}
\caption{The $A_2$ Hall algebra scattering diagram.} \label{a2}
\end{figure}
\end{eg}

\subsubsection{Genteel potentials}\label{GenteelSect}

We say that a quiver with potential $(Q,W)$ is \textbf{genteel} (or that $W$ is genteel) if the only incoming walls of $\f{D}^{\Hall}$ are 
\begin{align}\label{DHallIn}
\f{D}_{\In}^{\Hall}\coloneqq \{e_i^{\perp},\log 1_{\sst}(p^*(e_i))\}.
\end{align}
Theorems \ref{KSGS} and \ref{thm:Hallscattering} together imply the following:
\begin{lem}\label{GenteelProp}
If $(Q,W)$ is genteel, then $\f{D}_{\scat}^{\Hall}:=\scat(\f{D}_{\In}^{\Hall})$ exists and equals $\f{D}^{\Hall}$ (up to equivalence).
\end{lem}

It is expected (cf. \cite[Conj. 3.3.4]{WCS}) that for every $2$-acyclic quiver $Q$, a generic potential $W$ will be genteel (at least over $\ad(\f{g}^q)$ in the sense explained below).  The following is proved in \cite[\S 7.1]{DMan}:
\begin{lem}\label{AcyclicGenteel}
If $Q$ is acyclic --- or more generally, if the only cycles in $Q$ are composed of loops (i.e., $1$-cycles) --- then $(Q,0)$ is genteel.
\end{lem}

\begin{rmk}
On its face, Lemma \ref{AcyclicGenteel} is, in the cases without loops, the same as \cite[Lem. 11.5]{Bridge} (and the proof in \cite{DMan} is inspired by that in \cite{Bridge}).  However, \cite[\S 11.5]{Bridge} uses a slightly different and possibly flawed definition of genteel.  In \cite[Def. 11.3]{Bridge}, an object $E\in \rep(Q,W)$ is called self-stable if it is stable with respect to the stability condition $-p^*(\dim(E))$.  Then $(Q,W)$ is called genteel if the only self-stable objects are the simple objects $S_i$ for $i\in Q_0$.  Unfortunately, as pointed out to us by Lang Mou and acknowledged in \cite[arXiv v4]{Bridge}, it is not clear that this version of genteel really does imply the claim about incoming walls being as in \eqref{DHallIn}.  For this one would need to replace ``self-stable'' with ``self-semistable,'' but doing so results in other problems, e.g., acyclic examples which would fail to be genteel.  We have therefore taken the motivating property regarding incoming walls as our definition.
\end{rmk}

A potentially weaker (but for most purposes equally useful) version of genteel is as follows: we say that $(Q,W)$ is genteel over $\f{g}^{\f{i}}$ if, up to equivalence, the only incoming walls of $\f{D}^{\f{i}}$ are 
\begin{align*}%\label{DHallIn}
\f{D}_{\In}^{\f{i}}\coloneqq \{e_i^{\perp},\s{I}^{\f{i}}(\log 1_{\sst}(p^*(e_i)))\}.
\end{align*}
In general (even without genteelness), Theorem \ref{KSGS} guarantees the existence of $$\f{D}^{\f{i}}_{\scat}\coloneqq \scat(\f{D}^{\f{i}}_{\In}).$$
As with Lemma \ref{GenteelProp}, $W$ being genteel over $\f{g}^{\f{i}}$ means that $\f{D}^{\f{i}}_{\scat} = \f{D}^{\f{i}}$.  We note that genteel implies genteel over every $\f{g}^{\f{i}}$, and genteel over $\f{g}^{\f{i}}$ implies genteel over $\f{g}^{\f{i}'}$ for every $\f{i}'\supset \f{i}$.

\begin{prop}[\cite{Mou}, Cor. 1.2(i)]\label{GreenGenteel}
Let $(Q,W)$ be a quiver with potential (and no loops) such that $W$ is non-degenerate and $Q$ admits a green-to-red sequence.\footnote{See \cite[Def. 7.2]{DWZ} for the definition of a non-degenerate potential, and see \cite[Def. 3.1.1]{MulGreen} for the definition of a quiver admitting a green-to-red sequence (or see \cite[Def. 8.27]{GHKK} for the equivalent notion of a quiver with a ``large cluster complex'').  It is known (at least when allowing infinite potentials) that all quivers $Q$ without $2$-cycles admit non-degenerate potentials, cf. \cite[Cor. 7.4]{DWZ}.  In particular, $(Q,0)$ with $Q$ acyclic satisfy the hypotheses for Proposition \ref{GreenGenteel}.}  Then $(Q,W)$ is genteel over $\f{g}^{q}$ and $\f{g}^{\scl}$.  
\end{prop}

We note that a version of Proposition \ref{GreenGenteel} over $\ad(\f{g}^{\scl})$ (i.e., the quotient of $\f{g}^{\scl}$ by its center) was also proved in \cite[Thm. 1.2.2]{Qin}.  Also, \cite[Cor. 1.2(ii)]{Mou} proves that non-degenerate potentials for the Markov quiver (which does not admit a green-to-red sequence) are genteel over $\ad(\f{g}^{q})$ and $\ad(\f{g}^{\scl})$.

\begin{eg} \label{eg:initial}
For the $A_2$-quiver of Example \ref{ex:a2diag}, the simple representations are $\C \rightarrow 0$ and $0 \rightarrow \C$. Thus $\f{D}_{\In}^{\Hall}$ would be as in Figure \ref{initial}.  By Lemmas \ref{GenteelProp} and \ref{AcyclicGenteel}, we have that $\f{D}_{\scat}^{\Hall}$ exists and agrees with $\f{D}^{\Hall}$ from Figure \ref{a2}.

\begin{figure}[htb]
\centering
\begin{tikzpicture}
\draw
(-2,0) -- (2,0) node[right] {$\log\left(\sum_{k\geq 0} (0 \rightarrow \C)^{\oplus k}\right)$}
(0,-2) -- (0,2) node[left] {$\log\left(\sum_{k\geq 0} (\C \rightarrow 0)^{\oplus k}\right)$};
\end{tikzpicture}
\caption{The initial Hall algebra scattering diagram $\f{D}_{\In}^{\Hall}$ for the $A_2$-quiver.} \label{initial}
\end{figure}
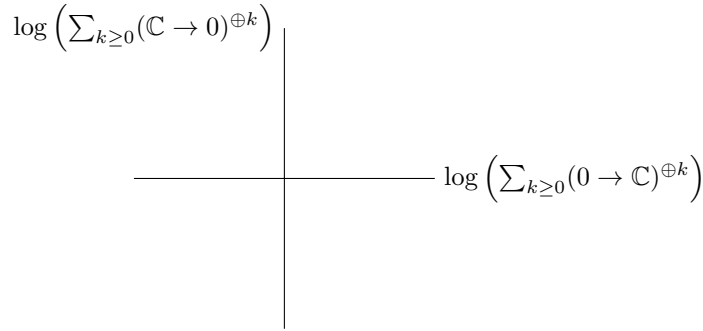
\end{eg}

\subsection{The initial Hall algebra scattering diagrams}\label{HallInitialSection}

We next wish to better understand the scattering functions of \eqref{DHallIn}.  We assume from now on that $Q$ contains no loops or oriented 2-cycles, although a generalization to cases with loops is possible --- cf. \cite[Prop. 7.7]{DMan} for a description of the incoming walls associated to vertices with a loop.

For each $i\in S$, we will find a nice expression for $\log 1_{\sst}(p^*(e_i))$ in terms of powers of $\kappa_i$.  We will need the \textbf{quantum dilogarithm}
\begin{align*}
\Psi_t(x):&=\sum_{k=0}^{\infty} \frac{t^{-k(k-1)/2} x^k}{(t-t^{-1})(t^2-t^{-2})\cdots(t^k-t^{-k})} \\
          &=\sum_{k=0}^{\infty} \frac{(tx)^k}{\prod_{j=1}^k (t^{2j}-1)},
\end{align*}
and the standard fact that $\log \Psi_t(x)=-\Li(-x;t)$, where
\begin{equation} \label{eq:li}
\Li(x;t)\coloneqq \sum_{k=1}^{\infty} \frac{x^k}{k(t^k-t^{-k})}.
\end{equation}

Denote
\begin{align*}
    f_i\coloneqq 1_{\sst}(p^*(e_i)) = \sum_{k=0}^{\infty} \delta_{ki}.
\end{align*}
By \eqref{delta-ki-kappa}, we can rewrite $f_i$ as\footnote{This expression seems to be well-known to experts, cf. \cite[\S 6.4]{KSmotivic}.}
\begin{align*}
        f_i=\sum_{k=0}^{\infty} \frac{\kappa_i^k}{\prod_{j=1}^k (q^j-1)} = \Psi_t\left(\frac{\kappa_i}{t}\right).
\end{align*}

Hence, using that $\log \Psi_t(x)=-\Li(-x;t)$, we find
\begin{align}\label{log-fi}
\log f_i = -\Li(-\kappa_i/t;t):&=\sum_{k=1}^{\infty} \frac{(-1)^{k-1}}{k(t^{k}-t^{-k})}\left(\frac{\kappa_i}{t}\right)^k  \\
                                       &=\sum_{k=1}^{\infty} \frac{(-1)^{k-1}}{k(q^{k}-1)}\kappa_i^k. \nonumber
\end{align}
We denote
\begin{align}\label{Rk}
    R_k\coloneqq \frac{(-1)^{k-1}}{k(q^{k}-1)}
\end{align}
so $\log f_i$ can be written as
\begin{align}\label{logfi}
    \log f_i = \sum_{k=1}^{\infty} R_k \kappa_i^k.
\end{align}

It follows immediately from \eqref{log-fi}, \eqref{Ikappa}, and Theorem \ref{KSGS} that applying $\s{I}_t$ to $\f{D}_{\scat}^{\Hall}$ produces the quantum cluster scattering diagrams of \cite[\S 4.2]{Man3}:

\begin{prop}\label{IgenteelProp}
Applying $\s{I}_t$ to $\f{D}_{\scat}^{\Hall}$ produces the scattering diagram $\f{D}^q_{\scat}\coloneqq \scat(\f{D}^q_{\In})$ over the quantum torus algebra, where 
\begin{align}\label{Dqin}
\f{D}^q_{\In}\coloneqq \{e_i^{\perp},-\Li(-z^{e_i},t)\}.
\end{align}
\end{prop}

Applying $\pi_{t\mapsto 1}$, it follows that $\s{I}$ applied to $\f{D}_{\scat}^{\Hall}$ yields $\scat(\f{D}^{\scl}_{\In})$, where 
\begin{align*}
    \f{D}^{\scl}_{\In}\coloneqq \{e_i^{\perp},-\Li(-z^{e_i})\}.
\end{align*}
Here, $\Li(x)\coloneqq \sum_{k=1}^{\infty} \frac{x^k}{k^2}$ is the classical dilogarithm.  This is precisely \cite[Lem. 11.4]{Bridge}.

\begin{rmk}
Instead of viewing $\f{g}^{\Hall}$ as $(t-t^{-1})^{-1}H_{\reg}(Q,W)$ with its commutator bracket, one might try to view it as simply $H_{\reg}(Q,W)$ with the Poisson bracket of \eqref{regPois}, cf. Remark \ref{PvC}.  In this version, instead of having $g=\log(1_{\sst}(\stab))$ in Theorem \ref{thm:Hallscattering}, one has $g=(t-t^{-1})\log(1_{\sst}(\stab))$.  In \eqref{logfi}, this corresponds to redefining $R_k$ to be $\frac{(-1)^{k-1}}{kt(1+q+q^2+\ldots+q^{k-1})}$.
\end{rmk}

\section{Scattering diagrams in terms of tropical disks}\label{TropSection}

\subsection{Tropical disks}\label{TropDisks}

We now introduce the tropical disks whose enumerations will be related to the scattering diagrams of \S \ref{ScatterSection}.  For now, our tropical disks will live in $L_{\bb{R}}\coloneqq L\otimes \bb{R}$ for an arbitrary finite-rank lattice $L$ (later we will take $L=\Lambda^{\vee}=M$). 

Let $\?{\Gamma}$ be the topological realization of a finite connected tree without bivalent vertices, and let $\Gamma$ denote the complement of all but one of its $1$-valent vertices.  Denote this remaining $1$-valent vertex by $V_{\infty}$, and denote the edge containing this vertex by $E_{\infty}$.  Let $\Gamma^{[0]}$, $\Gamma^{[1]}$, and $\Gamma^{[1]}_{\infty}$ denote the sets of vertices, edges, and non-compact edges of $\Gamma$, respectively.  Let $e_\infty\coloneqq \# \Gamma^{[1]}_\infty$.  Equip $\Gamma$ with a weighting $w:\Gamma^{[1]}\rar \bb{Z}_{>0}$, plus a marking $\epsilon:S \stackrel{\sim}{\rar} \Gamma^{[1]}_{\infty}$ for some index set $S$ with $\#S=e_{\infty}$.  For $s\in S$, we denote $E_s\coloneqq \epsilon(s)$. 

A \textbf{parametrized tropical disk} $(\Gamma,w,\epsilon,h)$ in $L_{\bb{R}}$ is data $\Gamma$, $w$, and $\epsilon$ as above, plus a proper continuous map $h:\Gamma\rar L_{\bb{R}}$ such that:
\begin{itemize}
\item For each $E\in \Gamma^{[1]}$, $h|_E$ is an embedding into an affine line with rational slope;
\item For any vertex $V$ and edge $E\ni V$, denote by $u_{(V,E)}$ the primitive integral vector emanating from $h(V)$ into $h(E)$.  For each $V\in \Gamma^{[0]}\setminus \{V_{\infty}\}$, the following \textbf{balancing condition} is satisfied:
\begin{align*}
\sum_{E\ni V} w(E) u_{(V,E)} =0.
\end{align*}
\end{itemize}
For unbounded edges $E_s\ni V$, we may denote $u_{(V,E_s)}$ simply as $u_{E_s}$ or $u_s$.  An \textbf{isomorphism} of parameterized tropical disks $(\Gamma,h)$ and $(\Gamma',h')$ is a homeomorphism $\Phi:\Gamma\rar \Gamma'$ respecting the weights and markings such that $h=h'\circ \Phi$.  A \textbf{tropical disk} is then defined to be an isomorphism class of parameterized tropical disks.  We will let $(\Gamma,h)$ denote the isomorphism class it represents, and we will often further abbreviate this as simply $\Gamma$ or $h$.

A \textbf{tropical ribbon} $\wh{\Gamma}$ is a tropical disk $(\Gamma,w,\epsilon,h)$ as above, together with the additional data of a cyclic ordering of the edges at each vertex.  A tropical disk or ribbon is called \textbf{trivalent} if every vertex other than $V_{\infty}$ is trivalent.
 
The \textbf{degree} $\Delta$ of a tropical disk $(\Gamma,w,\epsilon,h)$ is the map $\Delta:S\rar L$ given by
\begin{align*}
\Delta(s)=w(E_s)u_{E_s}.
\end{align*}
Let $\Flags(\Gamma)$ denote the set of flags $(V,E)$, $V\in E$, of $\Gamma$.  The \textbf{type} of a tropical disk is the data of $\Gamma$, $w$, and $\epsilon$, along with the data of the map $u:\Flags(\Gamma)\rar L$, $(V,E)\mapsto u_{(V,E)}$.  Note that the type of a tropical disk determines its degree.

Similarly, the \textbf{type} of a tropical ribbon is the data of the type of the associated tropical disk, plus the data of the ribbon structure, i.e., the data of the cyclic orderings at each vertex.

Let $\A\coloneqq (A_s)_{s\in S}$ be a tuple of affine-linear subspaces $A_s\subset L_{\bb{R}}$, each with rational slope.  We say a tropical disk $(\Gamma,w,\epsilon,h)$ \textbf{matches the constraint} $\A$ if $h(E_s) \subset A_s$ for each $s\in S$.

\subsection{Tropical degrees, constraints, and multiplicities associated to a scattering diagram}\label{ScatTrop}

We now combine the setup of \S \ref{TropDisks} with that of \S \ref{Scattering}.  Let $L=\Lambda^{\vee}$.  Let $\{e_i\}_{i\in I}$ be a finite collection of vectors in $\Lambda^+$, indexed by a set $I$. 
Suppose we have an initial scattering diagram $\f{D}_{\In}$ over $\f{g}$, with $\f{D}_{\In}$ having the form
\begin{align*}
    \f{D}_{\In}=\{(\f{d}_i,g_i)|i\in I\},
\end{align*}
where for each $i$, we have $\f{d}_i=e_i^{\perp}$ and
\begin{align}\label{gi}
    g_i=\sum_{w\geq 1} g_{i,w} \in \f{g}_{e_i}^{\parallel},
\end{align}
where $g_{i,w}\in \f{g}_{we_i}$.  Assume as in Theorem \ref{KSGS} that for each $i$, the terms $g_{i,w}$ pairwise commute.  We denote
\begin{align*}
    v_i=p^*(e_i)
\end{align*}
for each $i\in S$, so $-v_i$ is the direction of the wall $\f{d}_i$.

\subsubsection{Degrees and constraints}\label{DegConst}

Let $\ww\coloneqq (\ww_i)_{i\in S}$ be a tuple of weight vectors $\ww_i\coloneqq (w_{i1},\ldots,w_{il_i})$ with $0< w_{i1} \leq \ldots \leq w_{il_i}$, $w_{ij}\in \bb{Z}$.  For $\Sigma_{l_i}$ denoting the group of permutations of $\{1,\ldots,l_i\}$, let \[\Aut(\ww)\subset \prod_{i\in S} \Sigma_{l_i}\] be the group of automorphisms of the second indices of the weights $\ww_i$ which act trivially on $\ww$.  We also define
\begin{align*}
    l(\ww)=\sum_{i\in S} l_i.
\end{align*}

Associated to $\ww$, we consider the degree $\Delta_{\ww}:S_{\ww}\rar L$ given by
\begin{align}\label{S}
    S_{\ww}\coloneqq \{(i,j)|i\in I, j\in \{1,\ldots,l_i\}\},
\end{align}
and
\begin{align}\label{Deltaij}
    \Delta_{\ww}((i,j))=w_{ij}v_i.
\end{align}

For the associated constraints $$\A_{\ww}=(A_{ij})_{(i,j)\in S_{\ww}}$$ 
we take the affine-linear space $A_{ij}$ to be a generic translate of $e_i^{\perp}$.   Here, the translates for different pairs $(i,j)$ are generic relative to each other.  
We fix such a choice of $\A_{\ww}$ for each $\ww$.  Given $\delta>0$, let $\delta\A_{\ww}$ denote the constraints obtained from $\A_{\ww}$ by multiplying each $A_{ij}$ by $\delta$ (i.e., the distance from the origin is multiplied by $\delta$).

\subsubsection{Multiplicities}\label{MultSect}

For each $(i,j)\in S_{\ww}$, we denote
\begin{align}\label{gij}
    g_{ij}\coloneqq g_{i,w_{ij}}\in \f{g}_{w_{ij}e_i}.
\end{align}
Denote
\begin{align*}
    n_{\ww}\coloneqq \sum_{(i,j)\in S_{\ww}}w_{ij} e_i \in \Lambda.
\end{align*}

Now consider a trivalent tropical disk $\Gamma$ of degree $\Delta_{\ww}$.  We will denote $E_{(i,j)}$ simply as $E_{ij}$.  We view $\Gamma$ as flowing towards the univalent vertex $V_{\infty}$, and we use this flow to inductively associate an element $g_{E}\in \f{g}_{n_E}\subset \f{g}$ to each edge $E$ of $\Gamma$, where $n_E$ is an element of $\Lambda^+$ such that $p^*(n_E)\in L$ is the weighted tangent vector to $h(E)$ pointing in the direction opposite the flow.

To each of the source edges $E_{ij}$, we associate the element $g_{ij}$ from \eqref{gij} above.  Now consider a vertex $V\neq V_{\infty}$ with $E_1,E_2$ flowing into $V$ and $E_3$ flowing out of $V$, and suppose that for $i=1,2$, we already have associated elements $g_{E_i}\in \f{g}_{n_{E_i}}$.  By the balancing condition, we have  $n_{E_3}=n_{E_1}+n_{E_2}$.  Let us assume that the labelling of the edges $E_1,E_2$ is such that
\begin{align}\label{posV}
    \{n_{E_1},n_{E_2}\} \geq 0
\end{align}
(otherwise we re-label).  We then define
\begin{align*}
    g_{E_3}\coloneqq [g_{E_1},g_{E_2}]\in \f{g}_{n_{E_3}}.
\end{align*}
We now define the \textbf{multiplicity} of $\Gamma$ as
\begin{align*}
    \Mult(\Gamma) \coloneqq  g_{E_{\infty}} \in \f{g}_{n_{\ww}}.
\end{align*}

Now suppose that $\f{g}$ is a Lie subalgebra of the commutator algebra of a $\Lambda^+$-graded associative algebra $\f{A}$, i.e., we have an associative product such that
\begin{align*}
[g_1,g_2]=g_1g_2-g_2g_1.
\end{align*}

\begin{eg}\label{fA}
For $\f{g}=\f{g}^{\Hall}$, Remark \ref{PvC} says that we can take $$\f{g}^{\Hall}=(t-t^{-1})^{-1}\cdot H_{\reg}(Q,W) \subset \f{A}^{\Hall}\coloneqq H_{\reg}(Q,W)[(t-t^{-1})^{-1}].$$
Similarly, we can take 
\begin{align*}
    \f{g}^q=(t-t^{-1})^{-1} \bb{C}_t[N^{\oplus}] \subset \f{A}^{q}\coloneqq \bb{C}_t[N^{\oplus}][(t-t^{-1})^{-1}]. 
\end{align*} 
Moreover, for any $\f{i}$ such that $\f{i}^{\Skew}\subseteq \f{i}\subseteq \ker(\s{I}_t)$, since $(t-t^{-1})\notin \f{i}$, we can take
\begin{align*}
    \f{g}^{\f{i}}=(t-t^{-1})^{-1} H_{\reg}(Q,W)/\f{i} \subset \f{A}^{\f{i}}\coloneqq  (H_{\reg}(Q,W)/\f{i})[(t-t^{-1})^{-1}]. 
\end{align*}
However, $t-t^{-1}=0$ in $\f{g}^{\scl}$, so we cannot apply this localization in the classical setting.  Instead, we take $\f{A}^{\scl}$ to be the universal enveloping algebra of $\f{g}^{\scl}$.

Alternatively, the Poisson algebra $\bb{C}[N^{\oplus}]$ can be identified with a subalgebra of the module of log derivations  $\Theta(N^{\oplus})\coloneqq \C[N^{\oplus}]\otimes_{\bb{Z}} M$ as in \cite[\S 1.1]{GPS}.  Here, $z^n\otimes m$, typically denoted $z^n\partial_m$, is viewed as acting on $\C[N]$ via $z^{n'}\mapsto \langle n,m\rangle Z^{n+n'}\partial_m$.  The commutator of these derivations makes $\Theta(N^{\oplus})$ into a Lie algebra with bracket given by $$[z^{n_1}\partial_{m_1},z^{n_2}\partial_{m_2}]=z^{n_1+n_2}\partial_{\langle n_2,m_1\rangle m_2 - \langle n_1,m_2\rangle m_1}.$$
Let $\f{h}$ be the Lie subalgebra spanned by elements of the form $z^n\partial_m$ for $\langle n,m\rangle =0$.  Then $\bb{C}[N^{\oplus}]$ embeds into $\f{h}$ via $z^n \mapsto z^n\partial_{B(n,\cdot)}$.  Hence, instead of taking $\f{A}^{\scl}$ to be the universal enveloping algebra of $\bb{C}[N^{\oplus}]$, it is reasonable to take it to be the universal enveloping algebra of $\f{h}$ or $\Theta(N^{\oplus})$.  The latter is simply a log version of the Weyl algebra in $\rank(N)$ variables.  That is, we may view $\f{A}^{\scl}$ as an algebra of logarithmic differential forms.

We note that the usual classical multiplicities of tropical curves (as in correspondence theorems like those of \cite{NS}) can similarly be computed via iterated Lie brackets of polyvector fields, cf. \cite{MRudMult}.  Also, the quantum ribbon multiplicities computed using $\f{A}^{q}$ are related to certain counts of real curves, cf. \cite{Mikq}.
\end{eg}

For example, we can always take $\f{A}$ to be the universal enveloping algebra of $\f{g}$.  Alternatively, for $\f{g}^{\Hall}$ or $\f{g}^q$, we can produce such an $\f{A}$ using Remark \ref{PvC}.

Suppose that $\Gamma$ is equipped with a ribbon structure $\wh{\Gamma}$.  At each vertex $V\neq V_{\infty}$, let $E_1,E_2$ be the vertices flowing into $V$ and $E_3$ the vertex flowing out of $V$, and assume the cyclic ordering of the labelling $E_1,E_2,E_3$ agrees with the ribbon structure of $\wh{\Gamma}$ at $V$ (otherwise we re-label).  We say that the vertex $V\in \wh{\Gamma}^{[0]}$ is \textbf{positive} if the edges $E_1,E_2$, labelled in this way with respect to the ribbon structure, satisfy the condition \eqref{posV}.  Otherwise, we say $V$ is \textbf{negative}.

We now describe a method of inductively associating an element of $\f{g}_{n_E}\subset \f{A}$ to each edge $E$ of $\wh{\Gamma}$, this time denoting the elements by $g^{\rib}_{E}$. The vectors $n_E$ will be the same as before, but the elements $g^{\rib}_E$ will be different and will depend on the ribbon structure.  As before, we take $g_{E_{ij}}^{\rib}\coloneqq g_{ij}$ for the source edges.  But now, for $E_1,E_2$ the edges flowing into a vertex $V$, $E_3$ the edge flowing out of $V$, and the labelling agreeing with the ribbon structure at $V$, we define
\begin{align*}
    g^{\rib}_{E_3}\coloneqq \nu(V)g^{\rib}_{E_1}g^{\rib}_{E_2},
\end{align*}
where 
\begin{align*}
    \nu(V)\coloneqq \begin{cases} 1 &\mbox{if $V$ is positive}\\
    -1 &\mbox{if $V$ is negative,}\\
    \end{cases}
\end{align*}
and $g^{\rib}_{E_1}g^{\rib}_{E_2}$ is the associative product in $\f{A}$.  Finally, we define
\begin{align}\label{Multrib}
    \Mult^{\rib}(\wh{\Gamma})\coloneqq g^{\rib}_{E_{\infty}} \in \f{A}_{n_{\ww}}.
\end{align}

Alternatively, define
\begin{align*}
    \nu(\wh{\Gamma})\coloneqq \prod_{V\in \wh{\Gamma}} \nu(V).
\end{align*}
The ribbon structure induces an ordering of the unbounded edges of $\wh{\Gamma}$, starting with $E_{\infty}$ and then continuing with $E_{i_1j_1},\ldots,E_{i_{l(\ww)}j_{l(\ww)}}$.  Using the associativity of $\f{A}$, we can rewrite \eqref{Multrib} as
\begin{align}\label{MultribProd}
    \Mult^{\rib}(\wh{\Gamma}) = \nu(\wh{\Gamma})  g_{i_1j_1}g_{i_2j_2}\cdots g_{i_{l(\ww)}j_{l(\ww)}} \in \f{A}.
\end{align}

One easily sees the following: 

\begin{lem}\label{CurveRibLem}
For each $\Gamma$ as above,
\begin{align*}
    \Mult(\Gamma) = \sum \Mult^{\rib}(\wh{\Gamma}),
\end{align*}
where the sum is over all possible tropical ribbons $\wh{\Gamma}$ with underlying tropical curve $\Gamma$.
\end{lem}

Note that $\Mult(\Gamma)$ and $\Mult^{\rib}(\wh{\Gamma})$ are completely determined by the type $\tau$ of $\Gamma$ or $\wh{\Gamma}$, respectively.  We thus define the \textbf{multiplicity of a tropical disk or ribbon type} $\tau$ as the multiplicity of any of the tropical disks/ribbons of type $\tau$.

\subsection{Tropical ribbon counts and the consistent scattering diagram}\label{ScatFromRib}

We continue with the setup of \S \ref{ScatTrop}.  For each weight vector $\ww$, each $\delta>0$, and each $\stab\in L_{\bb{R}}$, let $\f{T}_{\ww,\delta}(\stab)$ denote the set of types of tropical disks of degree $\Delta_{\ww}$ which match the constraint $\delta\A_{\ww}$ and for which $h(V_{\infty})=\stab$.  For each $\epsilon>0$ and $\stab\in L_{\bb{R}}$, let $B_{\epsilon}(\stab)$ denote the open radius $\epsilon$ ball centered at $\stab$ (with respect to the Euclidean metric associated to any fixed choice of basis for $L$).  Let $\f{T}_{\ww}(\stab)$ denote the set of tropical disk types\footnote{We believe these can be interpreted as types of ``virtual tropical disks'' as defined in \cite[\S 5]{CPS}, so that the classical case of Theorem \ref{HallScatThm} can be viewed as a special case of \cite[Prop. 5.14]{CPS}.} $\tau$ such that, for any $\epsilon>0$ and all sufficiently small $\delta>0$, there exist $\stab'\in B_{\epsilon}(\stab)$ with $\tau \in \f{T}_{\ww,\delta}(\stab')$.  See Figure \ref{disk} for an example.

\begin{lem}\label{triv}
Recall our assumption that $\A_{\ww}$ is generic.  
For $\stab$ outside some locus of codimension $2$ (in particular, for $\stab$ general in the sense of \S \ref{Scattering}), every tropical disk type in $\f{T}_{\ww}(\stab)$ is trivalent.
\end{lem}
\begin{proof}
This follows from the correspondence between tropical disks and scattering walls in Lemma \ref{TropDkinfty} below.  More explicitly, 
consider a tropical disk $\Gamma$ of degree $\Delta_{\ww}$ matching the constraints $\delta\A_{\ww}$ for some $\delta>0$.  Consider the flow of $\Gamma$ towards $V_{\infty}$ as in \S \ref{MultSect}.  Suppose $E_1,\ldots,E_s$ flow into a vertex $V$ with $E_V$ flowing out, and suppose that $E_i$ lies in a generically translated affine hyperspace $A_{E_i}$ for $i=1,\ldots,s$.  Then $E_V$ lies in 
\begin{align*}
    A_{E_V}:=\left(\bigcap_i A_{E_i}\right) + \bb{R} v_{E_i},
\end{align*}
where $v_{E_V}=p^*(n_{E_V})$ is the direction of $E_V$.  In particular, \begin{align*}
    \codim(A_{E_V})=\left(\sum_{i=1}^s \codim(A_{E_i})\right)-1.
\end{align*}
For each unbounded edge $E_{ij}$, we can take $A_{E_{ij}}=A_{ij}$, which has codimension $1$.  It follows that, if there is a vertex of valence higher than three, then $h(V_{\infty})$ will necessarily lie in a generically determined translate of some rational-slope subspace of codimension at least $2$.  The assumption on $\stab$ then implies that $h(V_{\infty})$ cannot be in $B_{\epsilon}(\stab)$, and the result follows.
\end{proof}

Lemma \ref{triv} ensures that we can define the multiplicities of elements of $\f{T}_{\ww}(\stab)$ as in \S \ref{MultSect} whenever $\stab$ is outside some bad codimension $2$ locus (which will be the joints of a scattering diagram).  Define
\begin{align*}
    N(\stab)\coloneqq \sum_{\ww} \frac{1}{|\Aut(\ww)|}\sum_{\tau\in \f{T}_{\ww}(\stab)} \Mult(\tau)\in \wh{\f{g}}.
\end{align*}
Let $\f{T}^{\rib}_{\ww}(\stab)$ denote the set of tropical ribbons types $\wh{\tau}$ such that the associated tropical disk type $\tau$ is in $\f{T}_{\ww}(\stab)$.   By Lemma \ref{CurveRibLem}, we can express $N(\stab)$ as
\begin{align*}
    N(\stab)=\sum_{\ww} \frac{1}{|\Aut(\ww)|}\sum_{\wh{\tau}\in \f{T}^{\rib}_{\ww}(\stab)} \Mult^{\rib}(\wh{\tau}).
\end{align*}

Note that for each $n\in \Lambda^+$, the strict convexity of $\Lambda^+$ ensures that there are only finitely many $\ww$ such that $n=n_{\ww}$.  Furthermore, for each $\ww$, there are clearly only finitely many types of tropical disks of degree $\Delta_{\ww}$.  The well-definedness of $N(\stab)$ follows, assuming that we have already fixed $\A_{\ww}$.  The fact that the generic choice of $\A_{\ww}$ does not matter is part of the following theorem.

\begin{thm}\label{TropicalScattering}
Assume $\f{g}$ has Abelian walls.  Let $\f{D}=\scat(\f{D}_{\In})$, and consider $\stab\in L_{\bb{R}}\setminus \Joints(\f{D})$.  Up to equivalence, we may assume that $\f{D}$ has at most one wall $(\f{d},g_{\f{d}})\in \f{D}$ with $\stab\in \f{d}$.  If there is no such wall, then $N(\stab)=0$, and otherwise,
\begin{align*}
    g_{\f{d}}=N(\stab).
\end{align*}
\end{thm}

\subsection{Proof of Theorem \ref{TropicalScattering}}\label{TropScatProof}
Theorem \ref{TropicalScattering} is a modified version of \cite[Thm 3.7]{Man3}, or a refinement of some cases of \cite[Prop. 5.14]{CPS}.  The two-dimensional quantum version is \cite[Cor. 4.9]{FS}, and the two-dimensional classical version is \cite[Thm. 2.8]{GPS}.  The proof is similar in each case.  We repeat the setup here, following \cite[\S 3.2]{Man3}.

\subsubsection{Perturbing the scattering diagram}

\begin{dfn}\label{as}
For any scattering diagram $\f{D}$ over a Lie algebra $\f{g}$ with Abelian walls, the \textbf{asymptotic scattering diagram} $\f{D}_{\as}$ of $\f{D}$ is defined by replacing every wall $(n+\f{d},g_{\f{d}})\in \f{D}$ with the wall $(\f{d},g_{\f{d}})$.  Here, $\f{d}$ denotes a rational polyhedral cone (with apex at the origin) and $n \in N_{\bb{R}}$ translates this cone.
\end{dfn}

Now let $T$ denote the commutative polynomial ring $\bb{Z}[t_i|i\in I]$, and let $T_k\coloneqq  T/ \langle t_i^{k+1}|i\in I\rangle$.   
 Let $\f{D}_{\In,T_k}$ and $\f{D}_{\In,T}$ be the initial scattering diagrams over $\f{g} \otimes T_k$ and $\f{g}\otimes T$, respectively, given by replacing each $g_{\f{d}_i}=\sum_{j\geq 1} g_{ij}$ from $\f{D}_{\In}$ with $g'_{\f{d}_i}\coloneqq \sum_{j\geq 1} t_i^jg_{ij}$.   We will show that Theorem \ref{TropicalScattering} holds for $\f{D}_{T_k}\coloneqq \scat(\f{D}_{\In,T_k})$ for all $k$, hence for $\f{D}_T\coloneqq \scat(\f{D}_{\In,T})$.  Taking $t_i= 1$ for each $i$ then recovers the theorem for $\f{D}=\scat(\f{D}_{\In})$.
 
We have an inclusion of commutative rings
\begin{align*}
T_k&\hookrightarrow T'_k\coloneqq \bb{Z}[\?{u}_{ij}|i\in I, 1\leq j \leq k]/\langle \?{u}_{ij}^2|i\in I, 1\leq j \leq k\rangle\\
         t_i&\mapsto \sum_{j=1}^k \?{u}_{ij}.
\end{align*}
Using this inclusion to work in $\f{g}\otimes T'_k$, we have 
\begin{align}\label{Rwdiq}
g'_{\f{d}_i}= \sum_{w= 1}^k t_i^w g_{iw} = \sum_{w=1}^k \sum_{\#J=w} w! g_{iw} \?{u}_{iJ},
\end{align}
where the second sum is over all subsets $J\subset \{1,\ldots,k\}$ of size $w$, and
\begin{align*}
\?{u}_{iJ}\coloneqq \prod_{j\in J} \?{u}_{ij}.
\end{align*}

Consider our scattering diagram $\f{D}_{\In}=\{(\f{d}_i,g_i)|i\in I\}$ with $\f{d}_i=e_i^{\perp}$ and $g_i=\sum_{w\geq 1} g_{i,w}$ as in \eqref{gi}.  Applying the equivalence from Example \ref{eq}(1) in reverse and then perturbing the walls (i.e., translating the walls by some generic amount), we obtain a scattering diagram
\begin{align}\label{PertIn0}
\?{\f{D}}_k^0\coloneqq  \{(\f{d}_{iJ},w! g_{iw} \?{u}_{iJ}) | 1\leq w \leq k, J\subset \{1,\ldots,k\}, \#J=w\},
\end{align}
where $\f{d}_{iJ}$ is some generic translation of $\f{d}_i=e_i^{\perp}$.  Note that $\scat(\?{\f{D}}_k^0)_{\as} = \f{D}_{\In,T_k}$.

It will be useful for us to refine this setup a bit, working over a different commutative ring $\wt{T}_k$ defined by
\begin{align*}
    \wt{T}_k\coloneqq \bb{Z}[u_{iJ}|i\in I, J\subset \{1,\ldots,k\}]/\langle u_{iJ_1}u_{iJ_2}|J_1 \cap J_2 \neq \emptyset \rangle.
\end{align*}
Note that we have a surjective homomorphisms 
\begin{align}\label{pi}
  \wt{\pi}:\wt{T}_k \rar T'_k, \hspace{.25 in}  u_{iJ} \mapsto \?{u}_{iJ}.
\end{align}
Let $\f{D}_k^0$ denote the initial scattering diagram over $\f{g}\otimes \wt{T}_k$ defined as in \eqref{PertIn0}, but with the factors $\?{u}_{iJ}$ replaced by $u_{iJ}$, i.e.,
\begin{align}\label{PertIn}
\f{D}_k^0\coloneqq  \{(\f{d}_{iJ},w! g_{iw} u_{iJ}) | 1\leq w \leq k, J\subset \{1,\ldots,k\}, \#J=w\}.
\end{align}

\begin{eg} \label{eg:deform}
For $\f{D}_{\In}$ as in Example \ref{eg:initial}, the corresponding $\f{D}_k^0$ for $k=2$ may look like Figure \ref{deform}.

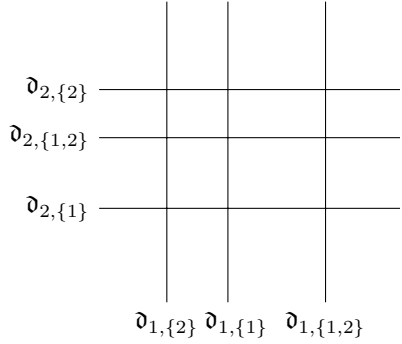
\begin{figure}[htb]
\centering
\begin{tikzpicture}
\draw
(-2,.83)node[left]{$\f{d}_{2,\{2\}}$} -- (2,.83) 
(-2,-.75)node[left]{$\f{d}_{2,\{1\}}$} -- (2,-.75) 
(-2,.19)node[left]{$\f{d}_{2,\{1,2\}}$} -- (2,.19) 
(-1.1,-2)node[below]{$\f{d}_{1,\{2\}}$} -- (-1.1,2)
(-.288,-2)node[below]{~$~\f{d}_{1,\{1\}}$} -- (-.288,2)
(1.01,-2)node[below]{$\f{d}_{1,\{1,2\}}$} -- (1.01,2);
\end{tikzpicture}
\caption{$\f{D}_2^0$ perturbing the $A_2$ initial scattering diagram of Figure \ref{initial}.} \label{deform}
\end{figure}
\end{eg}

Note that $\wt{\pi}$ takes the scattering functions of $\f{D}_k^0$ to those of $\?{\f{D}}_k^0$, and so the same will be true for the corresponding consistent scattering diagrams and their asymptotic versions.  We will write our walls in the form $(\f{d},g_{\f{d}}u_{\JJ_{\f{d}}})$, where $g_{\f{d}} \in \f{g}_{n_{\f{d}}}$ for some $n_{\f{d}}\in N^+$, $\JJ_{\f{d}}$ is a collection of pairwise-disjoint subsets of $I\times \{1,\ldots,k\}$ of the form $(i,J)$ for various $i\in I$ and $J\subset \{1,\ldots,k\}$, and
\begin{align}\label{uJ}
u_{\JJ_{\f{d}}}\coloneqq \prod_{(i,J)\in \JJ_{\f{d}}} u_{iJ}.
\end{align}

We now inductively produce a scattering diagram $\f{D}_k^{\infty}= \scat(\f{D}_k^0)$ from $\f{D}_k^0$ as follows: whenever two walls $(\f{d}_1,g_{\f{d}_1}u_{\JJ_{\f{d}_1}})$ and $(\f{d}_2,g_{\f{d}_2}u_{\JJ_{\f{d}_2}})$ intersect and satisfy $u_{\JJ_{\f{d}_1}}u_{\JJ_{\f{d}_2}}\neq 0$, we add a new wall $\f{d}(\f{d}_1,\f{d}_2)$ defined as follows:  assume $\{n_{\f{d}_1},n_{\f{d}_2}\}\geq 0$ (otherwise reorder), and then set
\begin{align}\label{Parents}
\f{d}(\f{d}_1,\f{d}_2)\coloneqq ((\f{d}_1\cap\f{d}_2)+\bb{R}_{\leq 0} p^*(n_{\f{d}_1}+n_{\f{d}_2}), [g_{\f{d}_1},g_{\f{d}_2}]u_{\JJ_{\f{d}_1}}u_{\JJ_{\f{d}_2}}).
\end{align}
This indeed terminates in finitely many steps and produces a consistent scattering diagram $\f{D}_k^{\infty}$, cf. \cite[\S 3.2.2 -- \S 3.2.3]{Man3} for details. We note that the consistency argument in \cite[\S 3.2.3]{Man3} requires the Abelian walls assumption.

\begin{dfn}
If $\f{d}=\f{d}(\f{d}_1,\f{d}_2)$, define $\Parents(\f{d})\coloneqq \{\f{d}_1,\f{d}_2\}$, and if $\f{d}\in \f{D}_k^0$, define $\Parents(\f{d})\coloneqq \emptyset$.  Recursively define $\Ancestors(\f{d})$ by $\Ancestors(\f{d})\coloneqq \{\f{d}\} \cup \bigcup_{\f{d}'\in \Parents(\f{d})} \Ancestors(\f{d}')$.  Define
\begin{align*}
\Leaves(\f{d})\coloneqq \{\f{d}'\in \Ancestors(\f{d})| \quad \f{d}' \mbox{~is the support of a wall in~} \f{D}_k^0 \}.
\end{align*}
\end{dfn}

\subsubsection{The tropical description of $\f{D}_k^{\infty}$}\label{TropDkinfty}

We continue to write $\JJ$ to denote a collection of pairwise-disjoint subsets of $I\times \{1,\ldots,k\}$ of the form $(i,J)$ for various $i\in I$ and $J\subset \{1,\ldots,k\}$.  Now, as in \S \ref{DegConst}, fix a weight vector $\ww\coloneqq (\ww_i)_{i\in I}$, $\ww_i\coloneqq (w_{i1},\ldots,w_{il_i})$ with $0< w_{i1} \leq \ldots \leq w_{il_i}$.  Let $\JJ_{\ww}$ denote the set of all possible $\JJ$ which can be written in the form $$\JJ=\{(i,J_{ij}):i\in I,j=1,\ldots,l_i\}$$ with $\#J_{ij}=w_{ij}$. 
Note that each $\JJ\in \JJ_{\ww}$ corresponds to a set of walls 
$$\f{D}_{k,\JJ}^0=\{\f{d}_{iJ_{ij}}\}_{(i,J_{ij})\in \JJ} \subset \f{D}_k^0,$$ and two choices of $\JJ$ correspond to the same $\f{D}_{k,\JJ}^0$ exactly if they are related by an element of $\Aut(\ww)$.  Given $\JJ$, let $\ww_{\JJ}$ denote the corresponding weight vector $\ww$ for which $\JJ\in \JJ_{\ww}$.

Let $\f{D}_{k,\JJ}^{\infty}$ denote the set of walls in $\f{D}_k^{\infty}$ whose leaves are precisely the walls of $\f{D}_{k,\JJ}^0$.  Note that, for  $\JJ\in \JJ_{\ww}$ and $(\f{d},g_{\f{d}}u_{\JJ}) \in \f{D}_{k,\JJ}^{\infty}$, we must have $g_{\f{d}} \in \f{g}_{n_{\ww}}$.  We will write $\f{T}_{\ww,\delta}(\stab,\A_{\JJ})$ to indicate $\f{T}_{\ww,\delta}(\stab)$ as in \S \ref{ScatFromRib} with the the representatives of the incidence conditions $\A_{\ww}$ chosen so that $A_{ij}=\f{d}_{iJ_{ij}}$.

\begin{lem}\label{ScatTropBij}
For every wall $(\f{d},g_{\f{d}}u_{\JJ})\in \f{D}_{k,\JJ}^{\infty}$ and every $\stab$ in the interior of $\f{d}$, there exists a unique tropical disk $h:\Gamma \rar L_{\bb{R}}$ in $\f{T}_{\ww_{\JJ},1}(\stab,\A_{\JJ})$ with $h(V_{\infty})=\stab$.  Furthermore, we have 
\begin{align}\label{BreakCoeff}
g_{\f{d}}=\Mult(\Gamma)\prod_{ij} (w_{ij}!).
\end{align}  
\end{lem}
\begin{proof}
We construct the tropical disk by starting at $h(V_{\infty})=\stab\in \f{d}$ and following $\f{d}$ in the direction $p^*(n_{\f{d}})$ until we reach a point $p\in \f{d}_1\cap \f{d}_2$, where $\{\f{d}_1,\f{d}_2\}=\Parents(\f{d})$.  The resulting segment is given weight $|n_{\f{d}}|$ (the index of $n_{\f{d}}$, i.e., $n_{\f{d}}$ equals $|n_{\f{d}}|>0$ times a primitive vector).  From $p$, extend the tropical curve in the directions $n_{\f{d}_1}$ and $n_{\f{d}_2}$ with weights $|n_{\f{d}_1}|$ and $|n_{\f{d}_2}|$, respectively, until reaching the boundaries of the walls $\f{d}_1$ and $\f{d}_2$.   The balancing condition at $p$ follows easily from \eqref{Parents} and the fact that commutators in $\f{g}$ respect the $N^+$-grading.  The process is repeated for each of these branches, and continues until every branch extends to infinity in some leaf.  This gives the desired tropical disk.  The formula for $g_{\f{d}}$ follows easily from \eqref{Parents} and the definition of $g_{\Gamma}$, noting that the $\prod w_{ij}!$ factor appears because of the fact that $g_{iw}$ is multiplied by $w!$ in the definition of $\f{D}_k^0$ in \eqref{PertIn}, and similarly for the $u_{\JJ}$ factor.
\end{proof}

\subsubsection{Proof of Theorem \ref{TropicalScattering}}

Given a weight vector $\ww$, let $|\ww_i|\coloneqq \sum_{j=1}^{l_i} w_{ij}$, and let $t^{\ww}=\prod_{i,j}t_i^{w_{ij}} = \prod_i t_i^{|\ww_i|}$.  Also, for $\JJ=\{(i,J_{ij})\subset I\times \{1,\ldots,k\}\}_{(i,j)\in S_{\ww}}$, let 
\begin{align*}
 \?{u}_{\JJ}\coloneqq \prod_{i,j} \?{u}_{iJ_{ij}}.
\end{align*}
We will use the following formula, cf. \cite[(45)]{Man3}:
\begin{align}\label{tww_sigma}
    t^{\ww} =  \sum_{\JJ\in \JJ_{\ww}}\left(\?{u}_{I_{\JJ}} \prod_{i,j} w_{ij}!\right).
\end{align}
For a scattering diagram $\f{D}$ and $\delta \in \bb{R}_{>0}$, let $\delta\f{D}$ denote the scattering diagram obtained by multiplying the supports of the walls of $\f{D}$ by $\delta$ (i.e., multiplying their distances from the origin by $0$). 

Now fix a point $\stab\in L_{\bb{R}}\setminus \Joints(\f{D}_{T_k})$.  Recall that $\f{D}_{T_k}=(\wt{\pi}(\f{D}_{k}^{\infty}))_{\as}$.   Hence, if $\stab \notin \supp(\f{D}_{T_k})$, then for sufficiently small $\delta>0$, no walls of $\delta\f{D}_{k}^{\infty}=\scat(\delta\f{D}_{k}^0)$ will intersect a small $\epsilon$-neighborhood of $\stab$.  So then by Lemma \ref{ScatTropBij}, no tropical disks representing a type in any $\f{T}_{\ww_{\JJ},1}(\stab,\A_{\JJ})$ will intersect such an $\epsilon$-neighborhood either, and so we obtain $N(\stab)=0$.

Now suppose $\stab \in \supp(\f{D}_k)$, and for convenience, use Example \ref{eq}(1) to combine all walls containing $\stab$ into a single wall $\f{d}$.  Then since $\f{D}_{T_k}=(\wt{\pi}(\f{D}_{k}^{\infty}))_{\as}$, we know that $g_{\f{d}}=\sum \wt{\pi}(g_{\JJ}u_{\JJ})$, where the sum is over all walls $(\f{d}_{\JJ},g_{\JJ}u_{\JJ})\in \f{D}_{k}^{\infty}$ such that for any $\epsilon>0$, there exists a $\delta>0$ for which $\delta \f{d}_{\JJ}$ intersects $B_{\epsilon}(\stab)$.  By Lemma \ref{ScatTropBij}, this is the same as 
\begin{align*}
    \wt{\pi}\left(\sum_{\ww} \frac{1}{|\Aut(\ww)|} \sum_{\JJ\in \JJ_{\ww}}\sum_{\tau\in \f{T}_{\ww}(\stab,\A_{\JJ})} \Mult(\tau)u_{\JJ}w_{ij}!\right),
\end{align*}
where here we write $\f{T}_{\ww}(\stab,\A_{\JJ})$ to indicate $\f{T}_{\ww}(\stab)$ for our particular choice of $\A_{\ww}$ as $\A_{\JJ}$ (since a priori $\f{T}_{\ww}(\stab)$ might depend on this choice).  Here we use our observation that two choices of $\JJ$ correspond to the same $\f{D}_{k,\JJ}^0$, hence the same $\f{D}_{k,\JJ}^{\infty}$, if and only if they are related by an element of $\Aut(\ww)$.

Now, note that for each $\ww$, $(\f{D}_k^0)_{\as}$ is symmetric with respect to permuting the elements of $\JJ_{\ww}$, i.e., for $\JJ_1,\JJ_2\in \JJ_{\ww}$, swapping the supports of $\f{d}_{i\JJ_1}$ and $\f{d}_{i\JJ_2}$ in \eqref{PertIn} does not affect $(\f{D}_k^0)_{\as}$. Hence,  $\sum_{\tau \in \f{T}_{\ww}(\stab,\A_{\JJ})} \Mult(\tau)$ is independent of $\JJ\in \JJ_{\ww}$, and so we obtain
\begin{align*}
    g_{\f{d}}=   \sum_{\ww} \frac{1}{|\Aut(\ww)|} \sum_{\tau\in \f{T}_{\ww}(\stab)} \Mult(\tau)\left(\sum_{\JJ\in \JJ_{\ww}}\?{u}_{\JJ}w_{ij}!\right).
\end{align*}
Finally, applying \eqref{tww_sigma} yields the desired result. \qed

\subsection{The Main Theorem}\label{MainThmSect}

We cannot apply Theorem \ref{TropicalScattering} directly to $\f{D}_{\scat}^{\Hall}$ because $\f{g}^{\Hall}$ is not skew-symmetric.  However, the theorem does apply to any of our $\f{D}_{\scat}^{\f{i}}\coloneqq \s{I}^{\f{i}}(\f{D}_{\scat}^{\Hall})$ for $\f{i}\supseteq \f{i}^{\Skew}$ as in \S \ref{HallScat}.  Here, we take the associative algebra $\f{A}^{\f{i}}$ to be as in Remark \ref{fA}, with $\f{A}^{\f{i}}$ meaning $\f{A}^{\scl}$ as Remark \ref{fA} in the case where $\f{i}=\ker(\s{I})$.

Given a weight vector $\ww$, define
\begin{align*}
    R_{\ww}\coloneqq \prod_{i,j} R_{w_{ij}},
\end{align*}
where we recall from \eqref{Rk} that $R_{k}\coloneqq \frac{(-1)^{k-1}}{k(q^k-1)}$.

Now, let us fix a quiver with potential $(Q,W)$ and consider the corresponding $\f{D}_{\scat}^{\Hall}$.  Consider a weight vector $\ww$ and a choice of $\wh{\tau}\in \f{T}_{\ww}^{\rib}(\stab)$.  Recall that the ribbon structure induces an ordering of the unbounded edges of $\wh{\tau}$, starting with $E_{\infty}$ and then continuing with $E_{i_1j_1},\ldots,E_{i_{l(\ww)}j_{l(\ww)}}$. Using \eqref{MultribProd} and \eqref{logfi}, we have
\begin{align*}
    \Mult^{\rib}(\wh{\tau})=\nu(\wh{\tau})R_{\ww} \kappa_{i_1}^{w_{i_1j_1}}\cdots \kappa_{i_{l(\ww)}}^{w_{i_{l(\ww)}j_{l(\ww)}}}.
\end{align*}
By Lemma \ref{CompSeriesLem}, we have
\begin{align}\label{kappaFlag}
    \kappa_{i_1}^{w_{i_1j_1}}\cdots \kappa_{i_{l(\ww)}}^{w_{i_{l(\ww)}j_{l(\ww)}}} = \f{Flag}(w_{i_1j_1}S_1,\ldots,w_{i_{l(\ww)}j_{l(\ww)}}S_{i_{l(\ww)}})=:\f{Flag}(\wh{\tau}),
\end{align}
where we write $w_{i_kj_k}S_{i_k}$ to indicate that the entry $S_{i_k}$ appears $w_{i_kj_k}$ times, and we neglect writing the data of the map to $\s{M}_{n_{\ww}}\subset \s{M}$.
 Finally, applying Theorem \ref{TropicalScattering} to the image under $\s{I}^{\f{i}}$, we obtain:
\begin{thm}\label{HallScatThm}
Let $\f{D}=\f{D}_{\scat}^{\f{i}}$ for $\f{i}\supset \f{i}^{\Skew}$, and consider $\stab\in L_{\bb{R}}\setminus \Joints(\f{D})$.  Up to equivalence, we may assume that $\f{D}$ has at most one wall $(\f{d},g_{\f{d}})\in \f{D}$ with $\stab\in \f{d}$.  If there is no such wall, then $N(\stab)=0$, and otherwise, $g_{\f{d}}=N(\stab)$, where $N(\stab)$ is defined as
\begin{align*}
 N(\stab)\coloneqq \sum_{\ww}\left( \frac{1}{|\Aut(\ww)|}\sum_{\wh{\tau}\in \wh{\f{T}}_{\ww}(\stab)} \nu(\wh{\tau})\s{I}^{\f{i}}(R_{\ww}\f{Flag}(\wh{\tau}))\right).
\end{align*}
\end{thm}

Combining Theorem \ref{HallScatThm} with Theorem \ref{thm:Hallscattering} and Lemma \ref{GenteelProp}, we immediately obtain the following:
\begin{thm}[Main result]\label{Main-DT-Trop-Thm}
Let $(Q,W)$ be a quiver with genteel potential over $\f{g}^{\f{i}}$ for some $\f{i}\supset \f{i}^{\Skew}$.  Let $\stab\in M_{\bb{R}}$ be general.  Then
\begin{align}\label{MainEqn}
    \s{I}^{\f{i}}(\log(1_{\sst}(\stab))) = \sum_{\ww}\left( \frac{1}{|\Aut(\ww)|}\sum_{\wh{\tau}\in \wh{\f{T}}_{\ww}(\stab)} \nu(\wh{\tau})\s{I}^{\f{i}}(R_{\ww}\f{Flag}(\wh{\tau}))\right).
\end{align}
\end{thm}
Theorem \ref{MainIntro} is the special case where $\f{i}=\ker(\s{I}_t)$.  The classical limit is the case where $\f{i}=\ker(\s{I})$.

\begin{rmk}\label{MultComputation}
While we find the expression of $\Mult^{\rib}(\wh{\tau})$ in terms of moduli of flags to be interesting, it is of course not generally simple to compute $\f{Flag}(\wh{\tau})$.  However, it is not difficult to describe the quantum and classical integrals of the terms $R_{\ww}\f{Flag}(\wh{\tau})$.

First, for the quantum cases, recall from \eqref{kappaFlag} that $\f{Flag}(\wh{\tau})$ arose as a product $\kappa_{i_1}^{w_{i_1j_1}}\cdots \kappa_{i_{l(\ww)}}^{w_{i_{l(\ww)}j_{l(\ww)}}}$.  From \eqref{IkappaEx}, $\s{I}_t(\kappa_{i_j})=tz^{e_{i_j}}\in \bb{C}_t[N^{\oplus}]$.  Hence, defining $R'_k\coloneqq \frac{(-1)^{k-1}}{k(t^k-t^{-k})}$ and $R'_{\ww}\coloneqq \prod_{i,j} R'_{w_{ij}}$, we have
\begin{align}\label{ItComputation}
    \s{I}_t(R_{\ww}\f{Flag}(\wh{\tau})) = R'_{\ww} \cdot z^{w_{i_1j_1}e_{i_1}}\cdots z^{w_{i_{l(\ww)}j_{l(\ww)}}e_{i_{l(\ww)}}}\in \f{A}^q.
\end{align}

Now for the classical version, recall that the map $\pi_{t\mapsto 1}:\f{g}^{q}\rar \f{g}^{\scl}$ takes $\frac{z^n}{t-t^{-1}}$ to $z^n$.  Let us embed $\f{g}^{\scl}$ into the Weyl algebra $\f{A}^{\scl}$ as in Example \ref{fA}, so $z^n$ becomes $z^n\partial_{B(n,\cdot)}$.  Let $R^{\scl}_k\coloneqq \frac{(-1)^k}{k^2}$ and $R^{\scl}_{\ww}\coloneqq \prod_{i,j} R^{\scl}_{w_{ij}}$.  We then obtain
\begin{align*}
    \s{I}(R_{\ww}\f{Flag}(\wh{\tau})) = R^{\scl}_{\ww}\cdot   z^{w_{i_1j_1}e_{i_1}}\partial_{B(w_{i_1j_1}e_{i_1},\cdot)}\cdots z^{w_{i_{l(\ww)}j_{l(\ww)}}e_{i_{l(\ww)}}}\partial_{B(w_{i_{l(\ww)}j_{l(\ww)}}e_{i_{l(\ww)}},\cdot)}\in \f{A}^{\scl}.
\end{align*}

We wrote Theorems \ref{HallScatThm} and \ref{Main-DT-Trop-Thm} in terms of tropical disk counts because we do not know a nice moduli-theoretic description of the tropical curve multiplicities.  However, there are again nice interpretations for the quantum and classical integrals.  Consider a tropical curve type $\tau\in \f{T}_{\ww}(\theta)$.  For each vertex $V$ of $\tau\setminus \{V_{\infty}\}$, let $u_1$ and $u_2$ be any two of the weighted tangent vectors of edges emanating from $V$.  Define $\Mult(V)\coloneqq |B(u_1,u_2)|$.  Then the classical multiplicity $\Mult(\tau)$ is given by 
\begin{align*}
    \Mult(\tau)=R^{\scl}_{\ww}\prod_V \Mult(V),
\end{align*}
and the quantum multiplicity $\Mult^q(\tau)$, by which we mean $\Mult(\tau)$ in the quantum cases, is given by
\begin{align*}
\Mult^q(\tau) = R'_{\ww}\prod_V [\Mult(V)]_t,
\end{align*}
where $[\Mult(V)]_t$ is defined as in \eqref{at}.
\end{rmk}

\begin{eg}\label{TermComputation}
Let us continue our ongoing example of the $A_2$ quiver $1 \rightarrow 2$ with $W=0$ as in Example \ref{ex:a2diag}.  Recall that in this case, $L \cong \Z^2$, $I=\{1,2\}$, and $B= \left(
\begin{array}{c c}
0 & -1 \\ 1 & 0
\end{array}
\right)$. Consider the weight-vector $\ww=((1,1),(1,1))$, so $\Aut(\ww)=2^2 = 4$.  For $\stab\in \bb{R}_{>0} (1,-1)$, a possible tropical disk type $\tau \in \f{T}_{\ww}(\stab)$ is illustrated in Figure \ref{disk}: 

\begin{figure}[htb]
\def\svgwidth{150pt}
    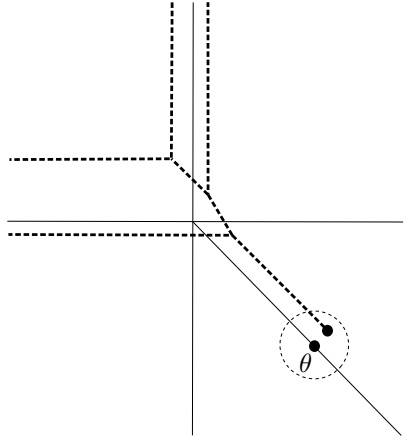
\caption{A tropical disk (given by the dashed segments) representing a type $\tau\in \f{T}_{\ww}(\stab)$.  The unbounded segments of the tropical curve are in generically specified lines.  The dashed circle around $\stab$ is an $\epsilon$-ball $B_{\epsilon}(\stab)$.  The solid rays are the support of the scattering diagram.} \label{disk}
\end{figure}

Since there are $3$ vertices, there are $2^3$ possible ribbon structures on this tropical disk.  One such ribbon structure $\wh{\tau}$ is illustrated in Figure \ref{ribbon}, namely, the ribbon structure for which $\nu(V)=1$ for each $V$.  

\begin{figure}[htb]
\def\svgwidth{150pt}
    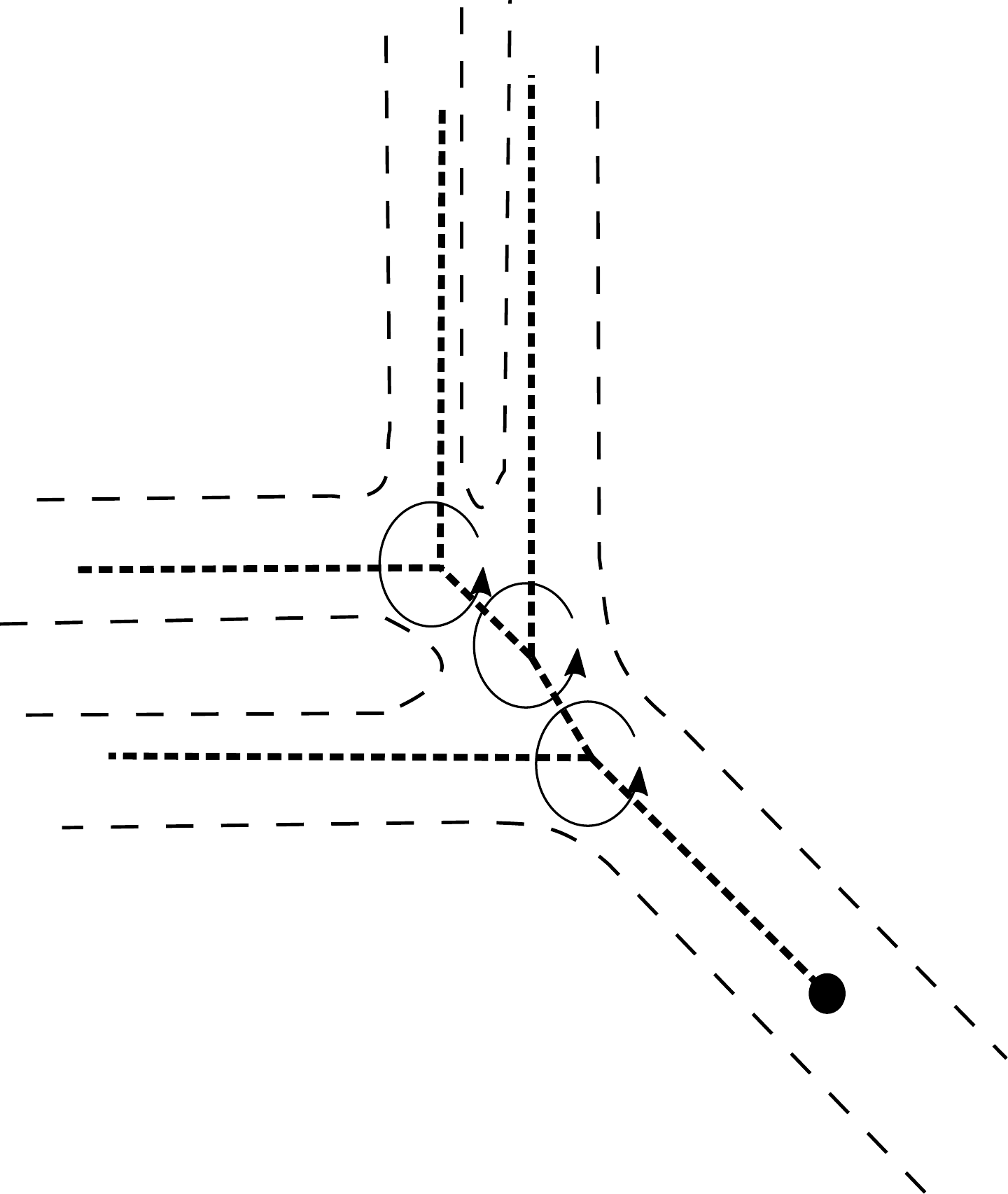
\caption{A tropical ribbon type $\wh{\tau}$ associated to the disk from Figure \ref{disk}.  The dashed curves around the tropical ribbon outline a topological realization of the ribbon, the induced orientation being clockwise.} 
\label{ribbon}
\end{figure}

The multiplicity $\Mult^{\rib}(\wh{\tau})$ is given by $\s{I}^{\f{i}}(\nu(\wh{\tau})R_{\ww}\kappa_1\kappa_1\kappa_2\kappa_2)$.  We easily see $\nu(\wh{\tau})=(-1)^3=-1$, and $R_{\ww}=\left(\frac{1}{q-1}\right)^3$.  The space $\f{Flag}(\wh{\tau})=\kappa_1\kappa_1\kappa_2\kappa_2$ is the space of composition series of the following form:
\[
(0 \rightarrow 0 ) \subset (\C \rightarrow 0) \subset (\C^2 \rightarrow 0) \subset (\C^2 \rightarrow \C ) \subset (\C^2 \rightarrow \C^2). 
\]
All maps $V\rightarrow W$ here, where $V, W$ are $\C$-vector spaces, are necessarily the $0$ map, and so by Lemma \ref{CompSeriesLem} we have $\f{Flag}(\wh{\tau})=\f{Flag}(S_2,S_2,S_1,S_1;\C^2 \stackrel{0}{\rightarrow}\C^2)\kappa_{\C^2 \stackrel{0}{\rightarrow}\C^2}$.  As in Example \ref{kappai}, since we work modulo automorphisms of $\C^2 {\xrightarrow{0}}\C^2$, 
there is only one flag, and its stabilizer group is $U_2(\C)^2$ (one copy of the unitary group $U_2(\C)$ for each $\C^2$).
Hence, we find
\begin{align*}
   \f{Flag}(\wh{\tau}) = \frac{1}{|U_2(\C)|^2} \kappa_{\C^2 \stackrel{0}{\rightarrow} \C^2} = \frac{1}{q^2}\kappa_{\C^2 \stackrel{0}{\rightarrow} \C^2}
\end{align*}

So the contribution of this tropical ribbon $\wh{\tau}$ to \eqref{MainEqn} is 
\begin{align}\label{contribution}
   \frac{1}{|\Aut(\ww)|} \nu(\wh{\tau})\s{I}^{\f{i}}(R_{\ww}\f{Flag}(\wh{\tau})) &= \left(\frac{-1}{4}\right) \s{I}^{\f{i}}\left(\frac{1}{(q-1)^3}\cdot\frac{1}{q^2}\kappa_{\C^2 \stackrel{0}{\rightarrow} \C^2}\right).
\end{align}
In the quantum case $\f{i}=\ker(\s{I}_t)$, we find using Proposition \ref{I-Ups} that $$\s{I}_t(\kappa_{\C^2 {\xrightarrow{0}} \C^2}) = \Upsilon(\mbox{pt})t^{\chi((2,2),(2,2))}z^{(2,2)} = t^4z^{(2,2)},$$ and so \eqref{contribution} becomes
\begin{align*}
    \frac{-z^{(2,2)}}{4(q-1)^3}.
\end{align*}
Alternatively, this may be computed using \eqref{ItComputation}.
\end{eg}

\section{Broken lines and theta functions}\label{ThetaSection}

\subsection{Definitions of broken lines and theta functions}\label{brokentheta}

Recall the notation and setup of \S \ref{Scattering}.  Fix a scattering diagram $\f{D}$.  
 Suppose we have a commutative ring $R$ and a $\Lambda$-graded $R$-algebra $A=\bigoplus_{\lambda\in \Lambda} A_{\lambda}$ with $A_0=R$ on which $\f{g}$ acts via $\Lambda$-graded $R$-algebra derivations.  We say that this action is \textbf{skew-symmetric} if $\f{g}_n\cdot A_{\lambda}=0$ whenever $\{n,\lambda\}=0$.  Let $\wh{A}$ denote the $(\Lambda^+)$-adic completion of $A$.  Note that $\wh{G}$ acts on $\wh{A}$ via $\Lambda$-graded $R$-algebra automorphisms.

\begin{dfn}\label{broken line}
Let $\lambda \in \Lambda\setminus \{0\}$, $\enpt \in \Lambda^{\vee}_{\bb{R}}\setminus \Supp(\f{D})$.  A  \textbf{broken line} $\gamma$ with  ends $(\lambda,\enpt)$
is the data of a continuous map $\gamma:(-\infty,0]\rar \Lambda^{\vee}_{\bb{R}}\setminus \Joints(\f{D})$, values $-\infty =:t_{-1}< t_0 \leq t_1 \leq \ldots \leq t_{\ell} = 0$, and for each $i=0,\ldots,\ell$, an associated homogeneous element $a_i \in A_{\lambda_i}$ for some $\lambda_{i} \in \Lambda\setminus \{0\}$, such that:
\begin{enumerate}[label=(\roman*), noitemsep]
\item $\lambda_0=\lambda$ and $\gamma(0)=\enpt$.
\item  For $i=0\ldots, \ell$, $\gamma'(t)=-p^*(\lambda_i)$ for all $t\in (t_{i-1},t_{i})$.  
\item $a_0=z^{\lambda}$.
\item For $i=0,\ldots,\ell-1$, $\gamma(t_i)\in \Supp(\f{D})$.  Let 
\begin{align}\label{kink}
g_i\coloneqq \prod_{\substack{(\f{d},g_{\f{d}})\in \f{D} \\ \f{d}\ni \gamma(t_i)}} \exp(g_{\f{d}})^{\sign\langle n_{\f{d}},p^*(\lambda)\rangle} \in \wh{G}.
\end{align}
I.e., $g_i$ is the $\epsilon\rar 0$ limit of the wall-crossing automorphism $\Phi_{\gamma|_{(t_i-\epsilon,t_i+\epsilon)}}$ defined in \eqref{WallCross} (using a smoothing of $\gamma$). Then $a_{i+1}$ is a homogeneous term of $g_i\cdot a_i$, other than $a_i$.
\end{enumerate}

The \textbf{theta function} $\vartheta_{\lambda,\enpt}\in \wh{A}$ is defined by
\begin{align*}
    \vartheta_{\lambda,\enpt}\coloneqq \sum_{\gamma} a_{\gamma},
\end{align*}
where the sum is over all broken lines $\gamma$ with ends $(\lambda,\enpt)$, and $a_{\gamma}$ denotes the element of $A$ associated to the last straight segment of $\gamma$.
\end{dfn}

 If $\f{g}$ is skew-symmetric and the action on $A$ is skew-symmetric, then \cite[Thm. 2.14]{Man3} (a refinement of \cite[Lemmas 4.7, 4.9]{CPS}) states the following:

\begin{lem}[The Carl-Pumperla-Siebert Lemma]\label{CPS}
Suppose $\f{g}$ is skew-symmetric with skew-symmetric action on $A$, and suppose $\f{D}=\scat(\f{D}_{\In})$ as in Theorem \ref{KSGS}.   Let $\gamma$ be a smooth path in $\Lambda^{\vee}_{\bb{R}}\setminus \Joints(\f{D})$ from $\enpt_1$ to $\enpt_2$, with $\enpt_1,\enpt_2\notin \Supp(\f{D})$.  Then for any $\lambda\in \Lambda$,
\begin{align*}
    \vartheta_{\lambda,\enpt_2}=\Phi_{\gamma,\f{D}} (\vartheta_{\lambda,\enpt_1}).
\end{align*}
\end{lem}

In any case, we have a copy $\wh{A}_{\enpt}$ of $\wh{A}$ and a collection of elements $\{\vartheta_{\lambda,\enpt}|\lambda \in \Lambda\}\subset \wh{A}_{\enpt}$ associated to every $\enpt\in \Lambda_{\bb{R}}^{\vee}\setminus \Supp(\f{D})$.  If $\f{D}$ is consistent, then the identifications of the $\wh{A}_{\enpt}$'s with $\wh{A}$ are all compatible with the path-ordered products. 
Furthermore, if Lemma \ref{CPS} holds, it says that the elements $\vartheta_{\lambda,\enpt}\in\wh{A}_{\enpt}$ are also compatible with the path-ordered products, thus giving a canonical collection of elements $\vartheta_{\lambda}\in \wh{A}$. 
We may therefore simply denote $\vartheta_{\lambda,\enpt}$ as $\vartheta_{\lambda}$.

\subsection{Hall algebra, quantum, and classical broken lines}\label{ClusterFlavors}

Take $\Lambda=N^{\prin}\coloneqq N\oplus M$, and take $\Lambda^{\oplus}\coloneqq (N^{\oplus},0)$. 
Denote $M^{\prin}\coloneqq (N^{\prin})^{\vee}=M\oplus N$.  Take $\{\cdot,\cdot\}$ to be the $\bb{Z}$-valued skew-symmetric form $B^{\prin}$ on $N^{\prin}$ defined via
\begin{equation} \label{eq:sform}
    B^{\prin}((n_1,m_1), (n_2,m_2)) = B(n_1,n_2) - \langle n_1,m_2\rangle + \langle n_2,m_1\rangle.
\end{equation}
We will write
\begin{align*}
    \pi^*:N&\rar M \\
     n &\mapsto B(\cdot,n), 
\end{align*}
while the map $p^*$ of \eqref{pistar} will be denoted by
\begin{align}\label{pprin}
p^{*,\prin}: N^{\prin}&\rar M^{\prin} \nonumber\\
             (n,m)&\mapsto B^{\prin}(\cdot,(n,m)) = (\pi^*(n)-m,n).
\end{align}
One can show that $B^{\prin}$ is unimodular, so the map $p^{*,\prin}$ is an isomorphism.

Now for our $\Lambda$-graded algebras, we take the following:
\begin{align*}
    A^{\Hall,\prin}\coloneqq H_{\reg}(Q,W)\otimes_{\bb{C}_{\reg}(t)} \bb{C}_{\reg}(t)[M],
\end{align*}
\begin{align*}
    A^{q,\prin}\coloneqq \bb{C}_t[N^{\oplus}]\otimes_{\bb{C}_{\reg}(t)} \bb{C}_{\reg}(t) [M],
\end{align*}
and
\begin{align*}
    A^{\scl,\prin}\coloneqq \bb{C}[N^{\oplus}]\otimes_{\bb{C}} \bb{C}[M].
\end{align*}
The algebra structure on $A^{\Hall,\prin}$ is determined by specifying that for $a_d\in H_{\reg}(Q,W)_d$ and $m\in M$, we have
\begin{align}\label{adzm}
    a_d*z^m = q^{-\langle d,m\rangle} z^m*a_d.
\end{align}
Similarly, for $A^{q,\prin}$ we specify that
\begin{align*}
    z^d*z^m=q^{-\langle d,m\rangle} z^m*z^d.
\end{align*}
Equivalently, $A^{q,\prin}$ is the quantum torus algebra $\bb{C}_t[\Lambda]$ with respect to the form $B^{\prin}$.  Finally, $A^{\scl,\prin}$ is just given the usual algebra structure, making it into $\bb{C}[\Lambda]$.

For the action of $\f{g}^{\Hall}$ on $A^{\Hall,\prin}$ we take the adjoint action, i.e., 
\begin{align*}
    g\cdot a\coloneqq [g,a]=ga-ag,
\end{align*}
using the natural inclusion of $\f{g}^{\Hall}$ into $(t-t^{-1})^{-1}\cdot A^{\Hall,\prin}$ to make sense of the multiplication.  Similarly for the action of $\f{g}^{q}$ on $A^{q,\prin}$.  The action of $\f{g}^{\scl}$ is the action with respect to the Poisson bracket, i.e.,
\begin{align*}
    z^d \cdot z^{(n,m)}\coloneqq \{z^{(d,0)},z^{(n,m)}\} = B^{\prin}((d,0),(n,m))z^{(n+d,m)}.
\end{align*}
One sees that the maps $\s{I}_t$, $\s{I}$, and $\pi_{t\mapsto 1}$ extend to homomorphisms between these algebras which commute with the corresponding Lie algebra actions.

We note that we could also define $A^{\f{i},\prin}$ for any other $\f{i}\supseteq \f{i}^{\Skew}$ by applying $\s{I}^{\f{i}}$ to $A^{\Hall,\prin}$.  The induced $\f{g}^{\f{i}}$-actions on $A^{\f{i},\prin}$ are skew-symmetric, thus yielding new examples of algebras for which Lemma \ref{CPS} holds.

Let $\Box$ represent $\Hall$, $\f{i}$, $q$, or $\scl$.  We can consider scattering diagrams in $M^{\prin}_{\bb{R}}$ over $\f{g}^{\Box}$.  We take $\f{D}_{\scat}^{\Box,\prin}\coloneqq \scat(\f{D}^{\Box,\prin}_{\In})$, where $\f{D}^{\Box,\prin}_{\In}$ is defined as in \eqref{DHallIn}, but with $(e_i,0)^{\perp}$ in place of $e_i^{\perp}$, and with $\log 1_{\sst}(p^*(e_i))$ replaced with its image under $\s{I}^{\f{i}}$, $\s{I}_t$ or $\s{I}$ if $\Box$ represents $\f{i}$, $q$ or $\scl$, respectively.  Note that the intersection of $\f{D}^{\Box}$ with $(M_{\bb{R}},0)\subset M^{\prin}_{\bb{R}}$ agrees with what we previously called $\scat(\f{D}^{\Box}_{\In})$.

\begin{rmk}\label{ProjectionRemark}
Note that all scattering walls have supports of the form $(n,0)^{\perp}$ for $n\in N$, so they are invariant under translation by $(0,N_{\bb{R}})$.  It follows that $\vartheta_{\lambda,\enpt}$ is invariant under translation of $\enpt$ by elements of $(0,N_{\bb{R}})$, and when enumerating broken lines, it suffices to consider their projections modulo $(0,N_{\bb{R}})$.
\end{rmk}

Note that $\f{g}^{\Box}$ and the action on $A^{\Box,\prin}$ are skew-symmetric if $\Box=\f{i},q$ or $\scl$, but typically not for $\Box=\Hall$.
With this setup and for $\Box=\Hall$, the broken lines with ends $(\lambda,\enpt)$ with $\enpt\in  (M \oplus N)_{\R}$ and $\lambda \in  \Lambda$ are precisely the \textbf{Hall algebra broken lines} discussed in \cite{mthesis}.  These will be examined in \S \ref{CPS-counterexample-section}.

We now briefly explain how the above theta functions for $\Box = \scl$ relate to those of \cite{GHKK}.  In the usual cluster algebras language, $z^m$ for $m\in M$ gives the  $A$ cluster variables, while $z^n$ for $n\in N$ gives the $X$ cluster variables.\footnote{The $A$ and $X$ notation is due to Fock and Goncharov, cf. \cite{FG1}, and was also used by \cite{GHKK}.  Some other authors use the Fomin-Zelevinsky \cite{FZ4} convention of denoting the $A$-variables by $X$ and the $X$-variables by $Y$.} In the principle coefficients setting, we have $z^{(m,n)} = \prod_i A_i^{m_i}X_i^{n_i}$, where $m=\sum_i m_i e_i^*$ and $n=\sum_i n_i e_i$.  The theta functions on $\s{A}^{\prin}$, $\s{A}$, and $\s{X}$ are obtained as follows: 
\begin{itemize}
    \item Allowing any $\lambda \in \Lambda$, the resulting theta functions $\vartheta_{\lambda}^{\prin}\coloneqq \vartheta_{\lambda}$ are the theta functions which \cite{GHKK} constructs on the cluster variety with principal coefficients $\s{A}^{\prin}$ (or rather, on some formal version of this in general).  The theta functions $\vartheta^{\prin}_{\lambda}$ for $\lambda\in (N^{\oplus})^{\vee}$ (i.e., the positive span of the vectors $e_i^*$) are the ones examined by Bridgeland \cite[Thm. 1.4]{Bridge}. 
    \item One obtains the theta function $\vartheta^{\s{A}}_{\lambda}$ on the cluster $\s{A}$-variety via the projection $(n,m) \mapsto m$ of $\vartheta^{\prin}_{\lambda}$ (i.e. setting all the $X$-variables in $\vartheta_{\lambda}$ equal to $1$), assuming that this projection is well-defined, i.e., that it converges. The middle cluster algebra of \cite{GHKK} is defined to be the span of all the $\vartheta^{\s{A}}_{\lambda}$ for which the convergence holds.  The corresponding elements $\lambda$ form a cone $\Xi\subset M$ which contains the Fock-Goncharov \cite{FG1} cluster complex $C$.  The $\vartheta^{\s{A}}_{\lambda}$ for $\lambda\in C$ give the cluster monomials.
    \item By applying a change of variables $z^{(m,n)} \mapsto z^n$ to $\{\vartheta_{(n,m)}|m=\pi^*(n)\}$, one obtains \cite{GHKK}'s theta functions 
$\vartheta^{\s{X}}_{n}$  for the $\s{X}$-space (or a formal version thereof).  Note that $m=\pi^*(n)$ implies $p^{*,\prin}((n,m))=(0,n)$, so this change of variables essentially amounts to applying $p^{*,\prin}$.
\end{itemize}

The theta functions for $\Box=q$ are among those considered in \cite{Man3}.  It was recently shown in \cite{DMan} that these form bases for the quantum cluster varieties (or formal versions thereof), thus giving quantum analogs for the results of \cite{GHKK}.

%and are expected to give bases for the quantum cluster varieties (or formal versions thereof).  These will be further explored in upcoming work of the second author with Ben Davison.

\subsection{Hall algebra theta functions and the CPS lemma}\label{CPS-counterexample-section}

As noted in \S \ref{ClusterFlavors},  $\f{g}^{\Hall}$ and its action on $A^{\Hall,\prin}$ typically fail to be skew-symmetric, and so \cite{Man3}'s proof of Lemma \ref{CPS} does not apply to Hall algebra broken lines.  In fact, we provide here a counterexample, thus showing that:
\begin{prop}\label{CPSfail}
The analog of Lemma \ref{CPS} does not generally hold for theta functions constructed from Hall algebra broken lines.
\end{prop} 
We note though that Hall algebra broken lines are still useful for understanding theta functions. For example, by studying Hall algebra broken lines and then integrating, we can understand the quantum or classical broken lines in terms of quiver Grassmannians, cf. \cite{mthesis}.

Recall from Remark \ref{ProjectionRemark} that we may compute theta functions using the images of broken lines under the projection $M_{\bb{R}}^{\prin}=M_{\bb{R}}\oplus N_{\bb{R}}\rar M_{\bb{R}}$.  \textbf{We will work in this projection throughout this subsection.}  Furthermore, the theta function we will consider will be of the form $\vartheta_{(0,m),\enpt}$ with $m\in \pi^*(N)$.  Suppose $a_i\in A_{\lambda_i}$ is the homogeneous element element attached to some straight segment of a broken line contributing to $\vartheta_{(0,m),\enpt}$.  Then $\lambda_i$ has the form $(n_i,m)$ for $n_i\in N$.  Using \eqref{pprin}, we see that the projection of $p^{*,\prin}(\lambda_i)$ modulo $(0,N)$ is $\pi^*(n_i)-m\in \pi^*(N)$.  Hence, by Definition \ref{broken line}(ii) (modulo $(0,N)$), we have
\begin{align}\label{gammaprime}
    \gamma'(t)=m-\pi^*(n_i)\in \pi^*(N)
\end{align}
for all $t$ in the corresponding straight segment of $\gamma$.  We thus obtain the following:
\begin{lem}\label{slice}
For $m\in \pi^*(N)$, all broken lines contributing to $\vartheta_{(0,m),\enpt}$ are contained in $\enpt+\pi^*(N_{\bb{R}})$.
\end{lem}

For our counterexample, we use the $A_3$-quiver $1 \rightarrow 2 \leftarrow 3$.  In the corresponding (standard) basis $e_1,e_2,e_3$ for $N$, the matrix for $B$ is \begin{align*}
    B=\left(\begin{matrix}
    0 & -1 & 0\\
    1 & 0 & 1\\
    0 & -1 & 0
    \end{matrix}\right).
\end{align*}
In general, the map $\pi^*:N\rar M$ takes $e_i$ to the $i$-th row of $B$.  In particular, we see that $\ker(\pi^*)$ is in this case generated by $e_1-e_3$, and $\Image(\pi^*)=(e_1-e_3)^{\perp}$, or the span of the first two rows of $B$ (viewed as vectors in the dual basis).

The walls of the initial scattering diagram $\f{D}^{\Hall}_{\In}$ (i.e., $\f{D}^{\Hall,\prin}_{\In}$ projected to $M_{\bb{R}}$) are $\f{d}_i\coloneqq (e_i^{\perp},1_{\sst}(\pi^*(e_i)))$.  Figure \ref{CPSfig} depicts a slice of the resulting consistent scattering diagram $\f{D}_{\scat}^{\Hall}$ (which exists and agrees with $\f{D}^{\Hall}$ by Lemmas \ref{GenteelProp} and \ref{AcyclicGenteel}), parallel to $\pi^*(N_{\bb{R}})$ (taking advantage of Lemma \ref{slice}), with the upward-pointing direction in the figure being parallel to $\pi^*(e_1)=\pi^*(e_3)=(0,1,0)$, and the leftward-pointing direction being parallel to $\pi^*(e_2)=(-1,0,-1)$.  In the figure, whenever two walls with attached scattering functions $g_{i_1},g_{i_2}$ collide with the corresponding $n_{i_1},n_{i_2}$ satisfying $B(n_{i_1},n_{i_2})>0$ (otherwise reorder), the picture is locally the same, up to a change of variables, as the picture in Example \ref{a2}.  Hence, consistency results in one new wall with attached scattering function given up to first order by $[g_{i_1},g_{i_2}]$.  Recall from \eqref{log-fi} that $\log(1_{\sst}(\pi^*(e_i)))=(q-1)^{-1}\kappa_i+(\mbox{higher order terms})$.  So in particular, the element of $\wh{\f{g}}^{\Hall}$ attached to $\f{d}_{12}$  is $(q-1)^{-2}[\kappa_2,\kappa_1]+(\mbox{higher order terms})$.

\begin{figure}[htb]
\def\svgwidth{200pt}
    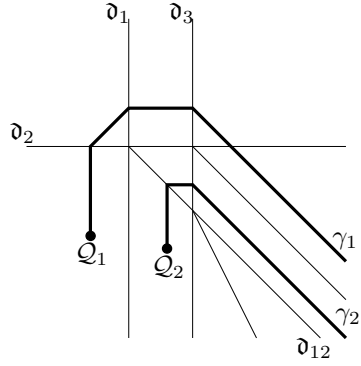
    \caption{Counterexample to the Carl-Pumperla-Siebert Lemma for Hall algebra broken lines. \label{CPSfig}}
\end{figure}

The bold lines in Figure \ref{CPSfig} represent broken lines with ends $(m,\enpt_1)$ and $(m,\enpt_2)$, where $m=\pi^*(e_1)+\pi^*(e_2)$, i.e., $m=(-1,1,-1)$ in the basis $e_1^*,e_2^*,e_3^*$ for $M$.  Here we keep in mind that by \eqref{gammaprime}, $\gamma'(t)=m$ on the first straight segment of $\gamma$.

By inspection, these are the only broken lines contributing to $\vartheta_{m,\enpt_1}$ or $\vartheta_{m,\enpt_2}$ whose attached element of $\wh{A}$ is in $A_{m+e_1+e_2+e_3}$ (the subscript denoting the degree in the $N^{\prin}$-grading).  We claim that for our counterexample, it suffices to check that these two attached elements are different from one another.  Indeed, while $\enpt_1$ and $\enpt_2$ do lie on opposite sides of $\f{d}_1$, the wall-crossing from $\enpt_2$ to $\enpt_1$ can only affect the grading in the $e_1$-direction.  So since there are no other broken lines with ends $(m,\enpt_2)$ whose final monomial has degree $(m+ke_1+e_2+e_3)$ for $k\in \bb{Z}$, the degree $(m+e_1+e_2+e_3)$ part of $\vartheta_{m,\enpt_2}$ after crossing $\f{d}_1$ must still be the final attached monomial of the bottom broken line from the figure.

Recall from \S \ref{ClusterFlavors} that when a broken line with attached element $a$ crosses a wall with attached element $g\in \wh{\f{g}}^{\Hall}$, the result of the action of $g$ on $a$ is $\exp [g,a]$.  In particular, if $g=\sum_{k\geq 1} g_k$ with $g_k\in \f{g}^{\Hall}_{kn}$ for some $n\in N^+$, then the action yields $[g,a]+\mbox{(higher order terms)}$.  Note that for each of the two broken lines in Figure \ref{CPSfig}, it is only the first-order terms of the scattering functions that contribute.  Also note that all the signs in the exponents as in \eqref{kink} are positive for the two broken lines under consideration.  We thus compute that the final attached element for the broken line $\gamma_1$ ending at $\enpt_1$ is 
\begin{align*}
    a_{\gamma_1}=(q-1)^{-3}[\kappa_2,[\kappa_1,[\kappa_3,z^m]]], 
\end{align*}
Similarly, the final attached element for the broken line $\gamma_2$ ending at $\enpt_2$ is
\begin{align*}
    a_{\gamma_2}&=(q-1)^{-3}[[\kappa_2,\kappa_1],[\kappa_3,z^m]] \\
    &=(q-1)^{-3}[\kappa_2,[\kappa_1,[\kappa_3,z^m]]]-(q-1)^{-3}[\kappa_1,[\kappa_2,[\kappa_3,z^m]]],
\end{align*}
where in the last step we applied the Jacobi identity.  So the difference between $a_{\gamma_2}$ and $a_{\gamma_1}$ is
\begin{align*}
    a_{\gamma_1}-a_{\gamma_2} = (q-1)^{-3}[\kappa_1,[\kappa_2,[\kappa_3,z^m]]]
\end{align*}

After applying $\s{I}^{\f{i}}$ for $\f{i}\supset \f{i}^{\Skew}$, the skew-symmetry of the brackets implies that {$[\kappa_2, [\kappa_3, z^m]]$} will vanish, and so this difference is indeed $0$ as implied by Lemma \ref{CPS}.  But in $A^{\Hall,\prin}$ this is not the case, as we will now check.  Using \eqref{adzm}, we compute that
\begin{align*}
    [\kappa_3,z^m]=(q^{-1}-1)z^m\kappa_3,
\end{align*}
and then
\begin{align}\label{z23}
    [\kappa_2,(q^{-1}-1)z^m\kappa_3]=(q^{-1}-1)z^m(q^{-1}\kappa_2\kappa_3-\kappa_3\kappa_2).
\end{align}

Let $\kappa_{23,f}$ denote the $\kappa$-element of the Hall algebra corresponding to the representation $(0\rightarrow \bb{C}\stackrel{f}{\leftarrow} \bb{C})$.  Up to isomorphism, $f$ here is either $0$ or $\id$.  Then using Lemma \ref{CompSeriesLem}, we compute
%\begin{align*}
%    \kappa_3\kappa_2=\kappa_{23,0}+(q-1)\kappa_{23,\id}
%\end{align*}
%and
%\begin{align*}
%    \kappa_2\kappa_3=\kappa_{23,0}.
%\end{align*}
\begin{align*}
    \kappa_2\kappa_3=\kappa_{23,0}+(q-1)\kappa_{23,\id}
\end{align*}
and
\begin{align*}
    \kappa_3\kappa_2=\kappa_{23,0}.
\end{align*}
Now the right-hand side of \eqref{z23} becomes:
%\begin{align*}
%    (1-q^{-1})(1-q)z^m(\kappa_{23,0}-\kappa_{23,\id})
%\end{align*}
\begin{align*}
    (q^{-1}-1)^2z^m(\kappa_{23,0}-\kappa_{23,\id})
\end{align*}
As a check, note that 
%$\s{I}_t(\kappa_{23,0}-\kappa_{23,\id})=0$
$\s{I}_t(\kappa_{23,0}-\kappa_{23,\id})=0$, so we do not violate Lemma \ref{CPS} after integrating.

Finally, we must check that that 
\begin{align*}
    [\kappa_1,z^m(\kappa_{23,0}-\kappa_{23,\id})]
\end{align*}
is nonzero.  Moving $\kappa_1$ past $z^m$ yields
\begin{align}\label{goalnonzero}
    [\kappa_1,z^m(\kappa_{23,0}-\kappa_{23,\id})] = z^m\left( q\kappa_1(\kappa_{23,0}-\kappa_{23,\id})  - (\kappa_{23,0}-\kappa_{23,\id})\kappa_1  \right).
\end{align}

Let $\kappa_{123,a,b}$ denote the $\kappa$-element corresponding to the representation $\bb{C} \stackrel{a}{\rightarrow} \bb{C} \stackrel{b}{\leftarrow} \bb{C}$.  Rather than completely computing $q\kappa_1(\kappa_{23,0}-\kappa_{23,\id})  - (\kappa_{23,0}-\kappa_{23,\id})\kappa_1$, 
let us just look at the coefficient of $\kappa_{123,0,0}$ in the product.  
The product $q\kappa_1\kappa_{23,0}$ yields a contribution of $q\kappa_{123,0,0}$, 
while the product  $-\kappa_{23,0}\kappa_1$ yields a contribution of $-\kappa_{123,0,0}$.  Thus, \eqref{goalnonzero} includes a term of the form
\begin{align*}
z^m(q-1)\kappa_{123,0,0}
\end{align*}
which will not cancel with any other terms.  In particular, \eqref{goalnonzero} is nonzero, as desired.  This proves Proposition \ref{CPSfail}. \qed

\bibliographystyle{alpha}  % Here the bibliography 		     %
\bibliography{mandel}        % is inserted.			     %
\index{Bibliography@\emph{Bibliography}}%

\end{document}

%% file: disk.pdf_tex
%% Creator: Inkscape 0.48.2, www.inkscape.org
%% PDF/EPS/PS + LaTeX output extension by Johan Engelen, 2010
%% Accompanies image file 'disk.pdf' (pdf, eps, ps)
%%
%% To include the image in your LaTeX document, write
%%   \input{<filename>.pdf_tex}
%%  instead of
%%   \includegraphics{<filename>.pdf}
%% To scale the image, write
%%   \def\svgwidth{<desired width>}
%%   \input{<filename>.pdf_tex}
%%  instead of
%%   \includegraphics[width=<desired width>]{<filename>.pdf}
%%
%% Images with a different path to the parent latex file can
%% be accessed with the `import' package (which may need to be
%% installed) using
%%   \usepackage{import}
%% in the preamble, and then including the image with
%%   \import{<path to file>}{<filename>.pdf_tex}
%% Alternatively, one can specify
%%   \graphicspath{{<path to file>/}}
%% 
%% For more information, please see info/svg-inkscape on CTAN:
%%   http://tug.ctan.org/tex-archive/info/svg-inkscape
%%
\begingroup%
  \makeatletter%
  \providecommand\color[2][]{%
    \errmessage{(Inkscape) Color is used for the text in Inkscape, but the package 'color.sty' is not loaded}%
    \renewcommand\color[2][]{}%
  }%
  \providecommand\transparent[1]{%
    \errmessage{(Inkscape) Transparency is used (non-zero) for the text in Inkscape, but the package 'transparent.sty' is not loaded}%
    \renewcommand\transparent[1]{}%
  }%
  \providecommand\rotatebox[2]{#2}%
  \ifx\svgwidth\undefined%
    \setlength{\unitlength}{405.07865445bp}%
    \ifx\svgscale\undefined%
      \relax%
    \else%
      \setlength{\unitlength}{\unitlength * \real{\svgscale}}%
    \fi%
  \else%
    \setlength{\unitlength}{\svgwidth}%
  \fi%
  \global\let\svgwidth\undefined%
  \global\let\svgscale\undefined%
  \makeatother%
  \begin{picture}(1,1.0939953)%
    \put(0,0){\includegraphics[width=\unitlength]{disk.pdf}}%
    \put(0.73576459,0.16066433){\color[rgb]{0,0,0}\makebox(0,0)[lb]{\smash{$\theta$}}}%
  \end{picture}%
\endgroup%

%% file: ribbon.pdf_tex
%% Creator: Inkscape inkscape 0.92.3, www.inkscape.org
%% PDF/EPS/PS + LaTeX output extension by Johan Engelen, 2010
%% Accompanies image file 'ribbon.pdf' (pdf, eps, ps)
%%
%% To include the image in your LaTeX document, write
%%   \input{<filename>.pdf_tex}
%%  instead of
%%   \includegraphics{<filename>.pdf}
%% To scale the image, write
%%   \def\svgwidth{<desired width>}
%%   \input{<filename>.pdf_tex}
%%  instead of
%%   \includegraphics[width=<desired width>]{<filename>.pdf}
%%
%% Images with a different path to the parent latex file can
%% be accessed with the `import' package (which may need to be
%% installed) using
%%   \usepackage{import}
%% in the preamble, and then including the image with
%%   \import{<path to file>}{<filename>.pdf_tex}
%% Alternatively, one can specify
%%   \graphicspath{{<path to file>/}}
%% 
%% For more information, please see info/svg-inkscape on CTAN:
%%   http://tug.ctan.org/tex-archive/info/svg-inkscape
%%
\begingroup%
  \makeatletter%
  \providecommand\color[2][]{%
    \errmessage{(Inkscape) Color is used for the text in Inkscape, but the package 'color.sty' is not loaded}%
    \renewcommand\color[2][]{}%
  }%
  \providecommand\transparent[1]{%
    \errmessage{(Inkscape) Transparency is used (non-zero) for the text in Inkscape, but the package 'transparent.sty' is not loaded}%
    \renewcommand\transparent[1]{}%
  }%
  \providecommand\rotatebox[2]{#2}%
  \newcommand*\fsize{\dimexpr\f@size pt\relax}%
  \newcommand*\lineheight[1]{\fontsize{\fsize}{#1\fsize}\selectfont}%
  \ifx\svgwidth\undefined%
    \setlength{\unitlength}{414.61567556bp}%
    \ifx\svgscale\undefined%
      \relax%
    \else%
      \setlength{\unitlength}{\unitlength * \real{\svgscale}}%
    \fi%
  \else%
    \setlength{\unitlength}{\svgwidth}%
  \fi%
  \global\let\svgwidth\undefined%
  \global\let\svgscale\undefined%
  \makeatother%
  \begin{picture}(1,1.18397088)%
    \lineheight{1}%
    \setlength\tabcolsep{0pt}%
    \put(0,0){\includegraphics[width=\unitlength,page=1]{ribbon.pdf}}%
  \end{picture}%
\endgroup%

%% file: CPScounterexample1.pdf_tex
%% Creator: Inkscape inkscape 0.92.3, www.inkscape.org
%% PDF/EPS/PS + LaTeX output extension by Johan Engelen, 2010
%% Accompanies image file 'CPScounterexample1.pdf' (pdf, eps, ps)
%%
%% To include the image in your LaTeX document, write
%%   \input{<filename>.pdf_tex}
%%  instead of
%%   \includegraphics{<filename>.pdf}
%% To scale the image, write
%%   \def\svgwidth{<desired width>}
%%   \input{<filename>.pdf_tex}
%%  instead of
%%   \includegraphics[width=<desired width>]{<filename>.pdf}
%%
%% Images with a different path to the parent latex file can
%% be accessed with the `import' package (which may need to be
%% installed) using
%%   \usepackage{import}
%% in the preamble, and then including the image with
%%   \import{<path to file>}{<filename>.pdf_tex}
%% Alternatively, one can specify
%%   \graphicspath{{<path to file>/}}
%% 
%% For more information, please see info/svg-inkscape on CTAN:
%%   http://tug.ctan.org/tex-archive/info/svg-inkscape
%%
\begingroup%
  \makeatletter%
  \providecommand\color[2][]{%
    \errmessage{(Inkscape) Color is used for the text in Inkscape, but the package 'color.sty' is not loaded}%
    \renewcommand\color[2][]{}%
  }%
  \providecommand\transparent[1]{%
    \errmessage{(Inkscape) Transparency is used (non-zero) for the text in Inkscape, but the package 'transparent.sty' is not loaded}%
    \renewcommand\transparent[1]{}%
  }%
  \providecommand\rotatebox[2]{#2}%
  \newcommand*\fsize{\dimexpr\f@size pt\relax}%
  \newcommand*\lineheight[1]{\fontsize{\fsize}{#1\fsize}\selectfont}%
  \ifx\svgwidth\undefined%
    \setlength{\unitlength}{155.19839287bp}%
    \ifx\svgscale\undefined%
      \relax%
    \else%
      \setlength{\unitlength}{\unitlength * \real{\svgscale}}%
    \fi%
  \else%
    \setlength{\unitlength}{\svgwidth}%
  \fi%
  \global\let\svgwidth\undefined%
  \global\let\svgscale\undefined%
  \makeatother%
  \begin{picture}(1,0.68577624)%
    \lineheight{1}%
    \setlength\tabcolsep{0pt}%
    \put(0,0){\includegraphics[width=\unitlength,page=1]{CPScounterexample1.pdf}}%
    \put(-0.00402317,0.41788822){\color[rgb]{0,0,0}\makebox(0,0)[lt]{\lineheight{0}\smash{\begin{tabular}[t]{l}$\f{d}_2$\end{tabular}}}}%
    \put(0.17685615,0.64486967){\color[rgb]{0,0,0}\makebox(0,0)[lt]{\lineheight{0}\smash{\begin{tabular}[t]{l}$\f{d}_1$\end{tabular}}}}%
    \put(0.29766926,0.64486967){\color[rgb]{0,0,0}\makebox(0,0)[lt]{\lineheight{0}\smash{\begin{tabular}[t]{l}$\f{d}_3$\end{tabular}}}}%
    \put(0.11886586,0.18577987){\color[rgb]{0,0,0}\makebox(0,0)[lt]{\lineheight{0}\smash{\begin{tabular}[t]{l}$\enpt_1$\end{tabular}}}}%
    \put(0.26384159,0.16161725){\color[rgb]{0,0,0}\makebox(0,0)[lt]{\lineheight{0}\smash{\begin{tabular}[t]{l}$\enpt_2$\end{tabular}}}}%
    \put(0.53751395,0.01017473){\color[rgb]{0,0,0}\makebox(0,0)[lt]{\lineheight{0}\smash{\begin{tabular}[t]{l}$\f{d}_{12}$\end{tabular}}}}%
    \put(0.60579208,0.21918546){\color[rgb]{0,0,0}\makebox(0,0)[lt]{\lineheight{0}\smash{\begin{tabular}[t]{l}$\gamma_1$\end{tabular}}}}%
    \put(0.61128383,0.07395821){\color[rgb]{0,0,0}\makebox(0,0)[lt]{\lineheight{0}\smash{\begin{tabular}[t]{l}$\gamma_2$\end{tabular}}}}%
  \end{picture}%
\endgroup%